\documentclass[11pt,reqno]{amsart}
\usepackage{amsmath}
\usepackage{amsfonts}
\usepackage{amsthm}
\usepackage{amssymb}
\usepackage{dsfont}
\usepackage{mathrsfs}
\usepackage{array}
\usepackage{latexsym}
\usepackage{verbatim}
\usepackage{a4wide}
\usepackage{enumerate}
\usepackage{hyperref}
\usepackage{lipsum}
\usepackage[all]{xy}
\usepackage{tikz-cd}  
\usepackage[title]{appendix}

\numberwithin{equation}{section}

\renewcommand{\thetheoremName}

\newcommand{\IB}{\mathbb{B}}
\newcommand{\IC}{\mathbb{C}}
\newcommand{\IF}{\mathbb{F}}

\newcommand{\IN}{\mathbb{N}}

\newcommand{\IP}{\mathbb{P}}
\newcommand{\IQ}{\mathbb{Q}}
\newcommand{\IR}{\mathbb{R}}
\newcommand{\IZ}{\mathbb{Z}}

\newcommand{\calC}{\mathcal{C}}

\newcommand{\calH}{\mathcal{H}}

\newcommand{\calO}{\mathcal{O}}

\newcommand{\calR}{\mathcal{R}}

\newcommand{\ip}{\mathfrak{p}}

\def\Vol{\mathrm{Vol}}

\def\GL{\mathrm{GL}}

\def\rad{\mathrm{rad}}

\def\SL{\mathrm{SL}}

\DeclareMathOperator\ord{ord}
\DeclareMathOperator{\pr}{proj}
\DeclareMathOperator{\argmin}{argmin}
\DeclareMathOperator{\Min}{min}

\DeclareMathOperator{\T}{T}
\DeclareMathOperator{\N}{N}
\DeclareMathOperator{\tr}{tr}
\DeclareMathOperator{\nr}{nr}

\newtheorem{theorem}{Theorem}[section]
\newtheorem*{theorem*}{Theorem}
\newtheorem*{corollary*}{Corollary}

\newtheorem{defn}[theorem]{Definition}
\newtheorem{lemma}[theorem]{Lemma}
\newtheorem{prop}[theorem]{Proposition}
\newtheorem{corollary}[theorem]{Corollary}

\theoremstyle{definition}
\newtheorem{definition}[theorem]{Definition}

\theoremstyle{remark}
\newtheorem{example}[theorem]{Example}

\newtheorem{remark}[theorem]{Remark}
\newtheorem{remarks}[theorem]{Remarks}

\thanks{This work was partially funded by the Swiss National Science Foundation (SNSF)}
\begin{document}
\title{Dense packings via lifts of codes to division rings}
\author{Nihar Gargava and Vlad Serban}

\bibliographystyle{alpha}

\begin{abstract}
We obtain algorithmically effective versions of the dense lattice sphere packings constructed from orders in $\IQ$-division rings by the first author. The lattices in question are lifts of suitable codes from prime characteristic to orders $\calO$ in $\IQ$-division rings and we prove a Minkowski--Hlawka type result for such lifts. Exploiting the additional symmetries under finite subgroups of units in $\calO$, we show this leads to effective constructions of lattices approaching the best known lower bounds on the packing density $\Delta_n$ in a variety of new dimensions $n$. This unifies and extends a number of previous constructions.  
\end{abstract}

\maketitle
\section{Introduction}
The sphere packing problem in $\mathbb{R}^n$ is concerned with maximizing the proportion of Euclidean space covered by a set of balls of equal radius and disjoint interiors. We will mostly be concerned with the \emph{lattice} sphere packing problem, where the balls are required to be centered at points on an $n$-dimensional lattice $\Lambda$. The proportion achieved by a particular lattice, called the packing density of $\Lambda$, is then given by 
$$\Delta(\Lambda):=\frac{\Vol(\IB_n(\lambda_1(\Lambda)))}{2^n\Vol(\Lambda)},$$
where $\lambda_1(\Lambda)$ denotes the shortest vector length in $\Lambda$,
$\IB_n(r)$ a ball of radius $r$ and $\Vol(\Lambda)$ denotes the covolume of the lattice. We also denote by $\Delta_n$ the supremum of lattice packing densities that can be achieved in $n$ dimensions. Its value is only known in a handful of dimensions, see for instance the summary \cite[1.5.]{Splag} as well as \cite{CohnKumar}. The density is achieved by highly symmetric lattices such as root lattices or the Leech lattice. Owing to highly celebrated results \cite{Hales2005APO,Marynadim8,Marynadim24}, some of these are even known to solve the general sphere packing problem. For arbitrary dimensions, the best known upper and lower bounds on $\Delta_n$ are however exponentially far apart as $n$ increases (see e.g., the survey article \cite{Cohn2016ACB} for more background). In this article we shall be concerned with lower bounds as $n$ increases and with giving effective constructions of lattices approaching these bounds.\par 
The classical Minkowski--Hlawka theorem \cite{Hlawka1943} states that $\Delta_n\geq 2\frac{\zeta(n)}{2^n}$, a bound which Rogers \cite{RogersExistence(Annals)} later improved by a linear factor to $\Delta_n\geq \frac{cn}{2^n}$ for $c=2/e\approx 0.74$. The constant was subsequently sharpened to $c=1.68$ by Davenport-Rogers \cite{DavenportRogers} and $c=2$ by Ball \cite{Ballbound} for all $n$. More recently, Vance \cite{VanceImprovedBounds} showed using lattices which are modules over the Hurwitz integers that one may take $c=24/e\approx 8.83$ and Venkatesh \cite{VenkateshBounds} showed that for $n$ large enough one may take $c=65963$. Moreover, by considering lattices from maximal orders in cyclotomic fields, Venkatesh was able to achieve for infinitely many dimensions the improvement $\Delta_n\geq \frac{n\log\log n}{2^{n+1}}$. The first author \cite{gargava2021lattice} then extended such results to lattices coming from orders $\calO$ in arbitrary $\IQ$-division algebras. This was achieved by proving a Siegel mean value theorem (see \cite{SiegelMVT,gargava2021lattice}) in this setting and exploiting the additional symmetries of the lattices under the group of finite order units in $\calO^\times$ to obtain dense packings. In particular, new sequences of dimensions such that $\Delta_n\geq \frac{c_1\cdot n\log\log(n)^{c_2}}{2^n}$ for constant $c_1,c_2>0$ are uncovered.

Lattices provide an important tool for coding, for instance for the Additive White Gaussian Noise (AWGN) model. For such applications it is often desirable to have lattices that are ``good'' in the sense of achieving high packing densities (see e.g., \cite{Loeliger97averagingbounds}). The sphere packing problem is moreover crucially connected to optimizing parameters of codes and energy optimization (see e.g.,\cite{Splag,UniversalOpt}). However, despite having Minkowski's lower bound for over a century, producing explicit families of lattices that achieve these asymptotic bounds in less than astronomical running time has proved elusive. Currently known polynomial time algorithms produce lattices whose densities are exponentially worse than these bounds \cite{litsyn1987constructive}. In this paper, we make the currently best known existential results for $\IQ$-division algebras effective by exhibiting finite sets of lattices which for large enough dimension must contain a lattice approaching the non-constructive lower bounds stated above. \par

Indeed, for orders $\calO$ in a $\IQ$-division algebra, we consider for suitable primes $p$ and for $t\geq 2$ the reduction map $\phi_p:\calO^t\to (\calO/p\calO)^t$ and may identify the quotient with a product of matrix rings over a finite field $\IF_q$. The sets of lattices $\mathbb{L}_p$ we consider are then re-scaled pre-images via $\phi_p$ of codes in $(\calO/p\calO)^t$ of a certain fixed $\IF_q$-dimension. Our first main result, Theorem \ref{thm:specificaverage}, is in effect a Siegel mean value theorem for these sets of lifts of codes valid for general finite dimensional $\IQ$-division algebras. We refer the reader to Section \ref{sec:three} for detailed statements, whereas some useful preliminary results on lattices from division algebras are established in Section \ref{sec:two}.\par 
In Section \ref{sec:four}, the extra symmetries of these lattices under finite subgroups of $\calO^\times$ is exploited to obtain (see Theorem \ref{thm:improvedbounds}): 
\begin{theorem}\label{thm:mainintro}
Let $A$ be a central simple division algebra over a number field $K$ with ring of integers $\calO_K$. Let $\calO$ be an $\calO_K$-order in $A$. Let $n^2=[A:K]$, $m=[K:\IQ]$ and let $t\geq 2$ be a positive integer. Let $G_0$ be a fixed finite subgroup of $\calO^\times$. Then there exists a lattice $\Lambda$ in dimension $n^2mt$ achieving 
$$\Delta(\Lambda)\geq\frac{\vert G_0\vert\zeta(mn^2t)\cdot t}{2^{mn^2t}\cdot e(1-e^{-t})}.$$
Moreover, there exists for any $\varepsilon>0$ an $\calO$-lattice $\Lambda_\varepsilon$ in dimension $n^2mt$ of packing density $$\Delta(\Lambda_\varepsilon)\geq(1-\varepsilon)\cdot \frac{\vert G_0\vert\zeta(mn^2t)\cdot t}{2^{mn^2t}\cdot e(1-e^{-t})}$$
which can be constructed. Indeed, $\Lambda_\varepsilon$ is obtained by applying Proposition \ref{prop:productminima} to a suitable sublattice of $\calO^t$ obtained as a pre-image via reduction modulo primes $\ip$ of $\calO_K$ of large enough norm of a code. The code in question is isomorphic to $k$ copies of simple left $\calO/\ip \calO$-modules for some $nt-t<k<nt$.
\end{theorem}
Note that Proposition \ref{prop:productminima} mentioned in the theorem is a version of a lemma of Minkowski extended to the division algebra setting. The theorem above is derived from Theorem \ref{thm:specificaverage} by mimicking an approach of Rogers \cite{RogersExistence(Annals)}, later used by Vance \cite{VanceImprovedBounds} and Campello \cite{CampelloRandom}. In this way, we combine nearly all of these existence results and generalize them in one theorem. \par 
In order to obtain the densest packings asymptotically, one therefore seeks families of orders $\calO$ with large finite unit groups $G_0\subset \calO^*$. Building on Amitsur's classification \cite{Amitsur1955FiniteSO} and following \cite{gargava2021lattice}, we give examples of such families. For instance, one may consider quaternion algebras over cyclotomic fields and hope to combine the improvements over the Minkowski-Hlawka bounds obtained by Vance and Venkatesh. However, due to a parity condition on the dimension, this is not quite the case and one obtains asymptotic lower bounds
$$\Delta_n\geq \frac{3\cdot (\log\log n)^{7/24}n\cdot(1+o(1))}{2^n}.$$
Still, the lower bound obtained exceeds the lower bound from cyclotomic fields $n\log\log n\cdot 2^{-(n+1)}$ in less than astronomical dimensions due to the improved constant. Moreover, one can exhibit lattices from non-commutative rings that achieve the same asymptotic density as the cyclotomic lattices (see Proposition \ref{prop:loglogimprovement}).
\par 
Finally, in Section \ref{sec:five} we establish effectivity by giving lower bounds on the norm of the prime $\ip$ so that packing densities such as in Theorem \ref{thm:mainintro} can be achieved, both by varying the division algebra and its dimension (Theorem \ref{thm:effective}) and by varying the rank of the $\calO$-lattices to obtain an effective version of Vance's results (Proposition \ref{prop:effectiveVance}). As an application, we obtain in Proposition \ref{prop:cycloquaternionseffective}: 
\begin{prop}
  Let $m_k=\prod_{\substack{p\leq k \text{ prime}\\2\nmid\ord_2p}}p$ and set $n_k:=8\varphi(m_k)$. Then for any $\varepsilon>0$ there is an effective constant $c_\varepsilon$ such that for $k>c_\varepsilon$ a lattice $\Lambda$ in dimension $n_k$ with density 
  $$\Delta(\Lambda)\geq (1-\varepsilon)\frac{24\cdot m_k}{2^{n_k}}$$
  can be constructed in $e^{4.5\cdot n_k\log(n_k)(1+o(1))}$ binary operations. 
  This construction leads to the asymptotic density of 
  $\Delta(\Lambda)\geq (1-e^{-n_k})\frac{3\cdot n_k(\log\log n_k)^{7/24}}{2^{n_k}}$. 
\end{prop}

We ought to stress that such effective lower bounds on the density are not the first of their kind but that there is a rather rich history of such results. Rush \cite{Rush1989}, building on work with Sloane \cite{RushnSloane}, recovered the Minkowski-Hlawka bound via coding-theoretic results such as the Varshamov-Gilbert bound and by lifting codes via the Leech-Sloane Construction A (see \cite[Chapter 5]{Splag}). The connection between random coding and such averaging results was further explicited by Loeliger \cite{Loeliger97averagingbounds}. This leads to families of approximate size $e^{n^2\log n}$ in which to search for lattices achieving the Minkowski-Hlawka bound. Gaborit and Z\'emor \cite{GaboritZemor2007} exploited additional structures to reduce the family size to $e^{n\log n}$. Finally, Moustrou \cite{MoustrouCodes} used a similar approach for cyclotomic lattices to obtain an effective version of Venkatesh's result. This approach was further formalized by Campello \cite{CampelloRandom}, where an example of such results for quaternion algebras is also mentioned. Our work thus owes a lot to these existing constructions. In particular, our approach is chiefly based on Moustrou's and Campello's work \cite{MoustrouCodes,CampelloRandom} and extends the scope of their results to division algebras, allowing symmetries from arbitrarily large non-commutative finite groups. We also note that the utility of codes from division rings is well-studied by coding theorists, see for example \cite{Berhuy2013AnIT, DucoatOggier, Vehkalahti2021}. \par

We hope that this article provides a useful addition to both the coding-theoretic and mathematical literature. The effective results we arrive at in Section \ref{sec:five} typically have a complexity of $e^{C\cdot  n\log n(1+o(1))}$, which is similar to \cite[Theorem 1]{MoustrouCodes}. However, the effective version of Vance's result (see Corollary \ref{cor:computationalVance}) has complexity $e^{1/4\cdot  n^2\log n(1+o(1))}$ and it should be similar for other constructions obtained by increasing the $\calO$-rank of the lattices.\par

The running times correspond to an exhaustive search through all the finitely many lattice candidates and it would be interesting to examine if one can further reduce the complexity of this search. It should be remarked that our results can still be used to quickly generate one of these random lattices in high dimensions that have prescribed symmetries and are expected to have large minimal vectors. Furthermore, our results hint towards the algebraic structures that one might look for in order to construct explicit families of lattices that are asymptotically ``good''.

\section{Preliminaries on division rings}\label{sec:two}
 In this section, we recall some definitions and results on central simple algebras and in particular division rings. The primary reference is Reiner's book \cite{reiner2003maximal}. 
 Let $\calO_K$ denote a Dedekind ring with quotient field $K$ and let $A$ denote a separable $K$-algebra. 
 \begin{definition}
   An $\calO_K$-order in $A$ is a subring $\calO$ of $A$ having the same identity element and such that $\calO$ is a \emph{full} $\calO_K$-lattice in $K$, i.e. $\calO$ is a finitely generated $\calO_K$-submodule of $A$ such that $K\cdot \calO=A$. 
 \end{definition}
 The existence of $\calO_K$-orders and maximal orders is easily shown. From now on let $\calO$ denote such an order. We first
review some results about ideals in $\calO$. Note that we shall typically state results for general $\calO_K$-orders, although simply dealing with maximal orders suffices for our applications. \subsection{Prime ideals}
\begin{defn}
A prime ideal of $\calO$ is a proper two-sided  ideal $\ip$ in $\calO$ such that $K\cdot \ip =A$ and such that for every pair of two sided ideals $S,T$ in $\Lambda$, we have that $S\cdot T\subset \ip$ implies $S\subset \ip$ or $T\subset \ip$. 
\end{defn}
For a prime $p$ of $\calO_K$ we shall denote by $\calO_p,A_p$ the localizations at $p$ of the $\calO_K$-order $\calO$ and of $A$ and by $\hat{\calO}_p,\hat{A_p}$ the respective completions. Finally let $\rad(R)$ denote the Jacobson radical of a ring $R$. It is easy to obtain (see \cite[Thm 22.3,22.4]{reiner2003maximal}) the characterization: 
\begin{theorem}\label{thm:primecorrespondence}
The prime ideals of an $\calO_K$-order $\calO$ coincide with the maximal two-sided ideals of $\calO$. If $\ip$ is a prime ideal of $\Lambda$, then $p=\ip\cap \calO_K$ is a non-zero prime of $\calO_K$, and $\overline{\calO}:=\calO/\ip$ is a finite dimensional simple algebra over the residue field $\calO_K/p$. Moreover, when $A$ is a central simple $K$-algebra, there is a bijection $\ip\leftrightarrow p$ between the set of primes of $\calO$ and of $\calO_K$, given by 
$$p=R\cap \ip\text{ and }\ip=\calO\cap \rad(\calO_p).$$
\end{theorem}
We note that when $A$ is a central simple $K$-algebra, the quotient $\calO/\ip$ is a simple Artin ring and hence a ring of matrices over a division ring. In particular, when $\calO_K/p$ is a finite field the quotient $\calO/\ip$ is isomorphic to a ring of matrices over a finite field. \par

We now summarize the behavior of $\calO$ and $A$ under localization as well as the splitting behavior for central simple algebras:
\begin{theorem}
Let $\calO$ be a maximal order in a central simple $K$-algebra $A$. Let $p$ denote a prime of $\calO_K$. Then the completion $\hat{A_p}$ is a central simple $\hat{K}_p$-algebra and $\hat{\calO}_p$ is a maximal order. Moreover:
\begin{enumerate}
    \item For almost every prime $p$ of $\calO_K$, we have that 
    $$\hat{A_p}\cong M_n(\hat{K}_p)$$
    with $n^2=[A:K]$ (\emph{split} or \emph{unramified} case). Moreover $p$ is split if and only if the corresponding prime ideal $\ip$ of $\calO$ as in Theorem \ref{thm:primecorrespondence} is just $p\calO$. \\
    \item The order $\Lambda=M_n(\hat{\calO}_p)$ is a maximal $\hat{\calO}_{K_p}$-order in $M_n(\hat{K}_p)$ having a unique maximal two-sided ideal $\pi_{\hat{K}_p}\Lambda$, where $\pi_{\hat{K}_p}$ is a prime element of the discrete valuation ring $\hat{\calO}_{K_p}$. The powers 
    $$(\pi_{\hat{K}_p}\Lambda)^t=\pi_{\hat{K}_p}^t\cdot\Lambda \text{ for }t=0,1,2,\ldots$$
    exhaust all the non-zero two-sided ideals of $\Lambda$ and any maximal $\hat{\calO}_{K_p}$-order is of the form $u\Lambda u^{-1}$ for $u\in \GL_n(\hat{K}_p)$. 
\end{enumerate}
\end{theorem}
\begin{proof}
These are well-known results, see e.g. \cite[Theorems 17.3, 32.1]{reiner2003maximal}.
\end{proof}

In what follows, we will thus use the same notation $\ip$ for the primes of $\calO_K$ and the corresponding prime in $\calO$ in the split case. 
\par
We can now establish the following lemma which will provide the necessary reduction maps in order to lift codes from characteristic $p$ to division rings: 
\begin{lemma}\label{lemma:primesexistence}
Let $K$ be a number field and $A$ a division algebra with center $K$ so that $[A:K]=n^2$. Let $\calO$ be an $\calO_K$-order in $A$. Then for all but finitely many primes $\ip$ of $\calO_K$, the quotient $\calO/\ip\calO$ is isomorphic to $M_n(\mathbb{F}_q)$, where $\calO_K/\ip\calO_K\cong \IF_q$. 
\end{lemma}
\begin{proof}
This follows from our previous considerations and in particular the fact that only finitely many primes of $A$ are ramified. Indeed, first assume that $\calO$ is a maximal order. Then we know that the completion satisfies
$$\hat{A}_\ip=A\otimes_K\hat{K}_\ip\cong M_n(\hat{K}_\ip)$$
and moreover that the order $\hat{\calO}_\ip$ in the local case is maximal and conjugate to $M_n(\hat{\calO}_{K_\ip})$. In particular, we then have
$$\calO/\ip\calO\cong M_n(\hat{\calO}_{K_\ip})/\rad(M_n(\hat{\calO}_{K_\ip}))\cong M_n(\hat{\calO}_{K_\ip}/\ip\hat{\calO}_{K_\ip})\cong M_n(\IF_q) .$$
When the order is not maximal, the same holds if one additionally avoids primes dividing the index of $\calO$ in a maximal order since their localizations will coincide at those primes. 
\end{proof}
If in the above one wishes to ensure $q=p$ is prime this can also typically be achieved for infinitely many primes by imposing that $p$ splits in $\calO_K$ and by the \v Cebotarev density theorem.\par

\subsection{Lattices from orders}
The central simple algebras $A$ are equipped with natural embeddings $A\hookrightarrow A\otimes_\IQ\IR$. The latter space being a semisimple $\IR$-algebra, it can be identified with a product of matrix rings over $\IR, \IC$ and $\mathbb{H}$ respectively, depending on the signature of $K$ and the splitting behavior of $A$ at infinity. An $\calO_K$-order then embeds as a lattice into this space. We establish some properties of these lattices in the next two subsections. 
\begin{lemma}
  Any semisimple $\mathbb{R}$-algebra $A$ admits an involution $(\ )^{*} : A \rightarrow A$ such that the following conditions are satisfied.
  \begin{itemize}
    \item For any $a,b \in A$, we have $(a b) ^{*}  = b^{*} a^{*}$.
    \item The trace yields a positive definite quadratic form $a \mapsto \T(a^{*} a  )$ on $A$, meaning that $\T(a^{*} a)$ is always non-negative and is zero only when $a =0$. Moreover, this quadratic form induces the inner product $\langle x, y\rangle =\T(x^* y)$ on $A$.
  \end{itemize}
  \label{le:involution}
  In particular, when $A$ is a division algebra over $\mathbb{Q}$, such an involution exists on $A \otimes_{\mathbb{Q}} \mathbb{R}$.
\end{lemma}
\begin{proof}
  See e.g. \cite[Corollary 35]{gargava2021lattice}
\end{proof}

We will denote $A \otimes_{\mathbb{Q}} \mathbb{R}$ by $A_{\mathbb{R}}$. Involutions with the properties as described in Lemma \ref{le:involution} will henceforth be called ``positive involutions''. An element $a \in A_{\mathbb{R}}$ such that $a^*=a$ and $x \mapsto \T(x^* ax)$ is a positive definite real quadratic form on $A_{\mathbb{R}}$ is called symmetric and positive definite. For instance, for any unit $a$ the element $a^*a$ is symmetric positive definite. \par
In practice, we will be considering $t\geq 2$ copies of orders $\calO$ in division algebras $A$ with center a number field $K$ and our lattices will be $\mathcal{O}^{t} \subseteq A^{t} \hookrightarrow (A_\mathbb{R})^{t}$. We will endow the space $A_\IR^t$ with the norm induced by the following quadratic form: 
  \begin{align}
    A_{\mathbb{R}}^{t} &  \rightarrow \mathbb{R} \\
    (x_1,x_2,\dots,x_t) & \mapsto  \sum_{i=1}^{t} \T(x_{i}^{*}ax_i),\label{eq:eva_norm}
  \end{align}
  where $(-)^*$ is a positive involution as defined above and $a \in A_{\mathbb{R}}$ is a symmetric positive definite element to be fixed later. This makes $A_{\IR}^t$ into an Euclidean space of real dimension $mn^2t$ and $\calO^t$ embeds as a lattice into that space.  

\subsection{Norm-trace inequality}
Recall that for a finite dimensional algebra $A$ over any field $k$ we have norm and trace functions $N_{A/k},T_{A/k}:A \rightarrow k$ given by the determinant and trace of the left multiplication maps respectively. It will be useful to establish some of their properties; for instance we later use the norm-trace inequality to give lower bounds on the Euclidean norm of certain lattice points via Lemma \ref{lemma:lowerboundbadpoints}.

\begin{lemma}\label{le:normtrace}
  Consider a finite dimensional semisimple $\mathbb{R}$-algebra $A_\IR$ together with a positive involution $(\ )^{*}$. Let $a \in A_\IR$ be a symmetric positive definite element and let $d = \dim_{\mathbb{R}} A_\IR$. Then $\N(a) > 0$, $\T(a) > 0$ and 
  \begin{align}
    \frac{1}{d}\T(a) \ge  \N(a)^{\frac{1}{d}}.
  \end{align}
\end{lemma}
\begin{proof}
  See \cite[Lemma 40]{gargava2021lattice}.
\end{proof}

\begin{corollary}
  In the same setting as above, we have for any $x \in A_\IR$ and for $a \in A_{\mathbb{R}}$ symmetric positive definite that 
  \begin{equation}
    \frac{1}{d} \T(x^{*} a x ) \ge \N(x)^{\frac{2}{d}} \N(a)^{\frac{1}{d}}.
  \end{equation}
  \begin{proof}
    Note that $x^* a x$ is symmetric and positive definite when $x$ is a unit in $A_\IR$. If $x$ is not a unit, then the right-hand side is trivial.
  \end{proof}
  
\end{corollary}
We remark that the inequalities above are sharp, equality being achieved by appropriate scalar matrices. 
For central simple $K$-algebras $A$, one can furthermore define the ``reduced norm'' and the ``reduced trace'' which we denote by $\nr_{A/K},\tr_{A/K}:A \rightarrow K$. Definitions and properties can be found in \cite[Section 9]{reiner2003maximal}. 

The following definition and ensuing lemma can also be found in \cite[9.13-14]{reiner2003maximal}.

\begin{definition}
  Suppose $A$ is a central simple $L$-algebra and $K \subseteq L$ is a subfield such that $[L:K] < \infty$. Then for each $a \in A$, we define the ``relative reduced trace'' $\tr_{A/K}:A \rightarrow K$ and ``relative reduced norm'' $\nr_{A/K}: A \rightarrow K$  as 
  \begin{align}
    \tr_{A/K} = \T_{L/K} \circ \tr_{A/L},\\
    \nr_{A/K} = \N_{L/K} \circ \nr_{A/L}.
  \end{align}
\end{definition}

\begin{lemma}
  The following holds for a central simple $L$-algebra $A$ and a subfield $K \subseteq L$ with $[L:K] < \infty$ for any $a \in A$:
  \begin{align}
    \T_{A/K}(a) = \sqrt{[A:L]} \tr_{A/K},  \\
    \N_{A/K}(a) = \nr_{A/K}(a)^{\sqrt{[A:L]}}.
  \end{align}
\end{lemma}

We may now establish the following lemmas:

\begin{lemma}
  Let $A$ be a division algebra over $\mathbb{Q}$ whose center is $K$ and $[A:K]= n^{2}$. Let $\mathcal{O} \subseteq A$ be a maximal order in the division algebra. Let $\ip$ be a prime ideal of $\calO_K$ for which $A$ splits and let $\mathbb{F}_q = \mathcal{O}_{K}/\ip$ denote the residue field. 
  Then the following diagram commutes:
\begin{equation}\label{nrddiagram}
\begin{tikzcd}
\mathcal{O} \arrow[d, "\phi_p"] \arrow[r, "\nr_{A/K}"] & \mathcal{O}_K \arrow[d, "\pi_\ip"] \\
\calO/\ip\calO\cong M_n(\mathbb{F}_q) \arrow[r, "\det"]                    & \mathbb{F}_q,                    
\end{tikzcd} 
\end{equation}
where the vertical maps designate reduction modulo $\ip$. 
  \label{lemma:nrdanddet}
\end{lemma}
\begin{proof}
  First note that $\nr_{A/K}(a) \in \mathcal{O}_{K}$ for each $a \in \mathcal{O}$ since $\nr_{A/K}(a)\in K$ and $\calO_K$ is integrally closed (see also \cite[(10.1)]{reiner2003maximal}). The reduced norm $\nr_{A/K}(a)$ may be computed as the determinant of the corresponding matrix in $M_n(E)$, where $E$ is a splitting field for $A$ (and is easily seen to be independent of that choice). We may in particular choose $E$ to be the $\ip$-adic completion $\hat{K}_\ip$ since by our assumption we have $A\otimes_K \hat{K}_\ip\cong M_n(\hat{K}_\ip)$. Therefore, in this case commutativity of the following diagram follows immediately: 
\begin{equation}\label{nrddiagram2}
\begin{tikzcd}
M_n(\calO_{\hat{K}_\ip}) \arrow[d, "\tilde{\phi_p}"] \arrow[r, "\det"] & \mathcal{O}_{\hat{K}_\ip} \arrow[d, "\pi_\ip"] \\
\calO_\ip/\ip\calO_\ip\cong M_n(\mathbb{F}_q) \arrow[r, "\det"]                    & \mathbb{F}_q,                    
\end{tikzcd} 
\end{equation}
where $\calO_p$ is the image of $\calO$ in $A\otimes_KK_\ip$. Now it is not necessarily the case that $\calO_\ip=M_n(\calO_{\hat{K}_\ip})$. However, in the complete local case $M_n(\calO_{\hat{K}_\ip})$ is a maximal $\calO_{\hat{K}_\ip}$-order and moreover all the maximal orders are conjugate (see \cite[Theorem 17.3]{reiner2003maximal}). Since the reduced norm is invariant under conjugation we may assume $\nr_{A/K}$ on $\calO$ factors through $M_n(\calO_{\hat{K}_\ip})$. Similarly, we can assume $\phi_p$ factors through $M_n(\calO_{\hat{K}_\ip})$. The commutativity of diagram \eqref{nrddiagram} thus follows from that of \eqref{nrddiagram2}. 

\end{proof}

\begin{lemma}\label{lemma:lowerboundbadpoints}
  Let $A$ be a $\IQ$-division algebra with center $K$, let $\calO$ denote a maximal order and let $\ip$ be a prime of $\calO_K$ at which $A$ is split, so that we have a reduction map $$\phi_p:\calO\to\calO/\ip\calO\cong M_n(\IF_q),$$
  with $n={\sqrt{[A:K]}}$. Let $(\ )^{*}: A_{\mathbb{R}} \rightarrow A_{\mathbb{R}}$ be a positive involution. If $x \in \mathcal{O} \setminus \{ 0\}$ (which we may identify with its image in $A_{\mathbb{R}}$) is such that $\phi_p(x)$ is a non-invertible matrix, then 
   \begin{equation}
       \|x\| \ge  \left( \sqrt{[A:\mathbb{Q}]} \N(a)^{\frac{1}{2[A:\IQ]}} \right)  q^{\frac{1}{\sqrt{[A:K]}[K:\mathbb{Q}]}} .
  \end{equation}
  where $a\in A_{\IR}$ is symmetric positive definite and $\|x\|^2:=\T(x^*ax)$ on $A_{\mathbb{R}}$.
\end{lemma}
  \begin{proof}
    We know from the norm-trace inequality that 
    \begin{align}
      \tfrac{1}{[A:\mathbb{Q}]} \T(x^{*} x ) \ge  \N(x)^{\frac{2}{{[A:\mathbb{Q}]}}} \N(a)^{\frac{1}{[A:\IQ]}}.
    \end{align}
    Since $\det \circ \phi_p(x) = 0$, we get by Lemma \ref{lemma:nrdanddet} that $\ip \mid \nr_{A/K}(x) $ and hence \begin{align*}
   &\N_{K/\mathbb{Q}}(\ip) \mid \N_{K/\mathbb{Q}} \circ \nr_{A/K}(x) \\
   \Rightarrow &N_{K/\mathbb{Q}}(\ip) \mid \nr_{A/\mathbb{Q}}(x)\\ \Rightarrow &N_{K/\mathbb{Q}}(\ip)^{\sqrt{[A:K]}} \mid \N_{A/\mathbb{Q}}(x) \\
   \end{align*}
    Since $\N_{A/\mathbb{Q}}(x)^2 \in \mathbb{Z}_{\ge 0}$ and $x \neq 0$, this gives us a lower bound on $\N_{A/\mathbb{Q}}(x)^2$. Thus we have that 
    \begin{align}
      \tfrac{1}{[A:\mathbb{Q}]} \T(x^{*}a x ) \ge  \N_{K/\mathbb{Q}}(\ip)^{\frac{2\sqrt{[A:K]}}{{[A:\mathbb{Q}]}}} \N(a)^{\frac{1}{[A:\IQ]}} .
    \end{align}
    which proves the claim given that $\N_{K/\mathbb{Q}}(\ip)=\vert\calO_K/\ip\vert=q$.
  \end{proof}
  We also record the following result: 
\begin{lemma}\label{lemma:packingradius}
Let $A$ be a central simple division $K$-algebra for $K$ a number field and let $\calO$ be a maximal $\calO_K$-order which we identify with the corresponding lattice in $A_\IR=A\otimes_\IQ\IR$. Then, with respect to any quadratic form $q_a(x)=\T(x^*ax)$ for a symmetric positive definite element $a \in A_\IR$, one may define the shortest vector length $\lambda_{1,q_a}$, Hermite parameter $\gamma_{q_a}$ and covering radius $\tau_{q_a}$, and they are subject to the following:
\begin{enumerate}
    \item The shortest vector length satisfies $\lambda_{1,q_a}(\calO)\geq \sqrt{[A:\IQ]}\cdot \N(a)^{1/(2[A:\IQ])}$.
    \item For any two-sided $\calO$-ideal I (by which we mean a full $\calO_K$-lattice in $\calO$), the Hermite parameter satisfies
    $$\gamma_{q_a}(I)\geq \frac{[A:\IQ]}{d(\calO/\IZ)^{1/[A:\IQ]}}$$
    \item The covering radius satisfies 
    $$\tau_{q_a}(\calO)\leq d(\calO/\IZ)^{1/[A:\IQ]}\cdot \left(\frac{\sqrt{[A:\IQ]}}{2\pi}+\frac{3}{\pi}\right)\cdot \N(a)^{-1/(2[A:\IQ])},$$
where $d(\calO/\IZ)$ denotes the discriminant of $\calO$ computed with respect to a $\IZ$-basis. 
\end{enumerate}
\end{lemma}
\begin{proof}The first part follows trivially from the norm-trace inequality. 
For the second statement, we know $I$ also yields a lattice in $A_\IR$ and, following \cite[Section 4]{BAYERFLUCKIGER2006305}, we set $\beta_{I,q_a}(x):=\frac{q_a(x)}{\det_{q_a}(I)^{1/[A:\IQ]}}$, where by definition $\det_{q_a}(I)$ is the determinant of a matrix of $q_a$ in a $\IZ$-basis of the ideal lattice $I$. Note that by definition $\gamma_{q_a}(I)=\min_{x\in I}\beta_{I,q_a}(x)$. By the norm-trace inequality, we have that 
$\N_{A/\IQ}(x)^2\leq (q_a(x)/[A:\IQ])^{[A:\IQ]}\cdot \N(a)^{-1}$. Therefore, since the discriminant of $\calO$ satisfies $d(\calO/\IZ)=\det_{q_{1}}(\calO)=\N(a)^{-1}\cdot \det_{q_a}(\calO)$, we obtain the inequality:
$$\N_{A/\IQ}(x)^2\leq \left(\frac{\beta_{\calO,q_a}(x)}{[A:\IQ]}\right)^{[A:\IQ]}\cdot d(\calO/\IZ)^2.$$
Moreover, letting $\mathcal{N}(I)$ denote the norm of an $\calO$-ideal as in \cite[Section 24]{reiner2003maximal} and $\mathcal{N}_\IQ(I)$ the number representing the fractional $\IZ$-ideal generated by $\N_{A/\IQ}(x)$ for $x\in I$ (see \cite[Theorem 24.9]{reiner2003maximal}), we have that 
$\det_{q_a}(I)=\det_{q_a}(\calO)\cdot \mathcal{N}_\IQ(I)^2$. 
Putting these together we deduce:
\begin{equation}\label{eq:betahermitebound}
    \beta_{I,q_a}(x)=\beta_{\calO,q_a}(x)\cdot \mathcal{N}_\IQ(I)^{-2/[A:\IQ]}\geq \frac{[A:\IQ]}{d(\calO/\IZ)^{1/[A:\IQ]}}\cdot\left(\frac{\N_{A/\IQ}(x)}{\mathcal{N}_\IQ(I)}\right)^{2/[A:\IQ]}\end{equation}
Taking the minimum over $x\in I$ of equation \eqref{eq:betahermitebound} yields the result, since for any $x\in I$ we have that $\N_{A/\IQ}(x)/\mathcal{N}_\IQ(I)\geq 1$.\par
For the third statement, we use a transference theorem: indeed, denoting by $\calO^\sharp$ the dual lattice, we have that 
$$\lambda_{1,q_a}(\calO^\sharp)\cdot \tau_{q_a}(\calO)\leq \frac{[A:\IQ]}{2\pi}+\frac{3\sqrt{[A:\IQ]}}{\pi}$$
by a refinement \cite[(1.9)]{Miller2019KissingNA} of a result of Banaszczyk \cite{Banaszczyk1993}.
But under the choice of inner product $\langle x,y\rangle=\T(x^*ay)$ corresponding to $q_a$, the dual lattice $\calO^\sharp$ is just $a^{-1}$ times the image under the involution $(-)^*$ of the inverse different $\mathfrak{D}(\calO/\IZ)^{-1}=\{x\in A:\tr_{A/\IQ}(x\cdot \calO)\subset \IZ\}$ and therefore has the same parameters. Given that $a^{-1}\mathfrak{D}(\calO/\IZ)^{-1}$ is a two sided (non-integral) $\calO$-ideal, by the second statement we deduce
$$\lambda_{1,q_a}(\calO^\sharp)=\lambda_{1,q_a}(a^{-1}\mathfrak{D}(\calO/\IZ)^{-1})\geq \sqrt{[A:\IQ]} \left(\frac{\det_{q_a}(a^{-1}\mathfrak{D}(\calO/\IZ)^{-1})}{d(\calO/\IZ)}\right)^{1/(2[A:\IQ])}.$$
Since $\det_{q_a}(a^{-1}\mathfrak{D}(\calO/\IZ)^{-1})=\det_{q_a}(\calO^\sharp)=\det_{q_a}(\calO)^{-1}$, we obtain that $$\lambda_{1,q_a}(\calO^\sharp)\geq\sqrt{[A:\IQ]}\cdot d(\calO/\IZ)^{-1/[A:\IQ]}\cdot \N(a)^{1/(2[A:\IQ])}.$$
The result follows by the transference theorem. 
\end{proof}



We now have the tools to tackle in the next section the first main result of this paper. 

\section{A general averaging result for lifts}\label{sec:three}

Again, let $A$ denote a central simple $K$-algebra for $K$ a number field and assume $A$ is a division ring. We set $n^2=[A:K]$ and $m=[K:\IQ]$. Let $\calO$ denote an order in $A$. The purpose of this section is to prove a discrete version of a Siegel mean value theorem (see \cite{SiegelMVT,VenkateshBounds,gargava2021lattice}) for $\calO$-lattices obtained via lifts of codes in characteristic $p$ for increasingly large primes $p$. In order to facilitate discussion and comparison with previous results, we will state a more abstract and flexible version of the main averaging result and later explain the most convenient choices of parameters. \par

Assume we are given for a fixed integer $t\geq 2$ a family of surjective homomorphisms 
$$\phi_p: \calO^t \to \calR_p^t,$$
indexed by primes $p$, where $\calR_p$ is a finite $\IF_p$-algebra of fixed $\IF_p$-rank $d_\calR$ and where the sizes of $p,\calR_p$ are unbounded in the family. We also fix an embedding $i: A\hookrightarrow \IR^{n^2m}$ which identifies $\calO^t$ with a lattice in $\IR^{n^2mt}$. 
Assume we are given subsets $U_p\subset \calR_p^t$ and sets of $\calR_p$- submodules $\calC_k$ of $\calR_p^t$ of fixed $\IF_p$-dimension $d(k)< t\cdot d_\calR$ (which we refer to as generalized $k$-dimensional codes) such that the following two conditions hold: 
\begin{enumerate}
    \item ($U$-balancedness): For any fixed $p$ in the family and $x\in U_p$, the number of $k$-dimensional codes in $ \mathcal{C}_k$ containing $x$ is constant (we denote this constant $L_{U_p}$). 
    \item (non-degeneracy): There exists a constant $c_U>0$ and $s>\frac{d_\calR\cdot t-\max_k d(k)}{mn^2t}$ such that 
    $$\|i(x)\|\geq c_U\cdot p^s \text{ for any nonzero }x\in \phi_p^{-1}(\calR_p^t\setminus U_p).$$
    We also require the mild condition that $\lim_{p\to\infty}\frac{\vert U_p\vert}{\vert R_p\vert^t}=1$.
    \end{enumerate}
We can then prove an averaging result for lifts in this general setup. Denote within this setup the set of scaled lattices
$$\mathbb{L}_p=\{\beta\phi_p^{-1}(C):C\in \calC_k\},$$ 
where $\beta$ is chosen such that all the lattices in $\mathbb{L}_p$ have volume $V:=\Vol(\calO^t)$, ergo $\beta=(\frac{p^{d(k)}}{\vert R_p\vert^t})^{1/n^2mt}$. 
Finally, we define following \cite[Def 2]{CampelloRandom}: 
\begin{definition}
  A function $f:\IR^d\to \IR$ is called semi-admissible if $f$ is Riemann-integrable and there exist positive constants $b,\delta>0$ such that 
  $$\vert f(\mathbf{x})\vert\leq\frac{b}{(1+\|\mathbf{x}\|)^{d+\delta}}\text{ for all }\mathbf{x}\in \IR^d.$$
\end{definition}
We then have the following: 
\begin{theorem}\label{thm:mainaverage}
Let $\calC_k$ be the set of codes of fixed $\IF_p$-dimension  $d(k)$ and let $f:\IR^{n^2mt}\to \IR$ be a semi-admissible function for $t\geq 2$.
Under the setup and notations above, assume the $U$-balancedness and non-degeneracy conditions are satisfied. Then, provided $d_\calR\cdot t>d(k)>t\cdot(d_\calR-s(n^2m))$, we have: 
$$\lim_{p\to\infty}\mathbb{E}_{\mathbb{L}_p}\left(\sum_{x\in (\beta\phi_p^{-1}(C))'}f(i(x))\right)\leq (\zeta(n^2mt)\cdot V) ^{-1}\cdot \int_{\IR^{n^2mt}}f(x)dx,$$
where $p$ ranges over the primes in the family and where for $\Lambda\in \mathbb{L}_p$ we denote by $\Lambda'$ the primitive vectors in $\Lambda$ (=shortest vectors of the lattice on the real line they span). 
\end{theorem}
We note that we set up the non-degeneracy condition so that in the very least the largest dimension $d(k)$ possible for the appropriate notion of $k$-dimensional code which is strictly contained in $\calR_p^t$ satisfies the condition $d_\calR\cdot t>d(k)>t\cdot(d_\calR-s(n^2m))$ and so the result is not vacuous.
\begin{proof}
We split the expected value of the lattice sum into two parts: 
First, we show that the expected value of the term 
$$\sum_{x\in (\beta\phi_p^{-1}(C))', \phi_p(x/\beta)\in (\calR_p^t\setminus U_p)}f(i(x))=\sum_{x\in (\phi_p^{-1}(C))', \phi_p(x)\in (\calR_p^t\setminus U_p)}f(\beta i(x))$$
tends to zero in absolute value as $p\to \infty$. By the non-degeneracy condition,  we can bound 
\begin{equation}\label{eq:lowerboundnorminthm}
\|\beta i(x)\|\geq \beta \cdot c_u\cdot p^s=c_u\cdot p^{(d(k)-d_\calR\cdot t)/(mn^2t)}p^s
\end{equation}

which under our assumption on $d(k)$ becomes arbitrarily large as $p\to \infty$. Since $f$ decays rapidly at infinity we get for each individual lattice in $\mathbb{L}_p$ that the sum converges to $0$ as $p\to \infty$ by dominated convergence, and therefore idem for the average.\par
The remaining terms can be bounded on average via $U$-balancedness. Indeed, let $g:R_p^t\to \IR^+$ denote any function. We have that 
\begin{align*}
    \mathbb{E}_\calC[\sum_{c\in \calC\cap U_p}g(c)]&=\sum_{x\in U_p}\mathbb{E}_\calC[g(x)\mathbf{1}_C(x)]\\
    &=\sum_{x\in U_p} g(x) \frac{L_{U_p}}{\vert \calC\vert}\\
    &\leq \sum_{x\in U_p} g(x) \frac{p^{d(k)}}{\vert U_p \vert},
\end{align*}
where we use $U$-balancedness as well as a counting argument for the inequality and where the expected value is taken as average over all codes in $\calC$. We deduce the inequality 
$$ \mathbb{E}[\sum_{x\in(\phi_p^{-1}(C))',\phi_p(x)\in U_p}f(\beta i(x))]\leq\frac{p^{d(k)}}{\vert U_p \vert}\sum_{x\in (\calO^t)'} f(\beta i(x)).$$

By the M\"obius inversion formula (and exchanging summation and limits by dominated convergence when $f$ does not have bounded support), the latter equals 
$$\sum_{r=1}^\infty \frac{\mu(r)}{r^{n^2mt}}\sum_{x\in \calO^t\setminus \{0\}}\frac{p^{d(k)}}{\vert U_p \vert}r^{n^2mt}f(r\beta i(x)).$$
The result now follows: first, note that we have for fixed $r\in \IN$ by approximation of the Riemann integral of $f$ that 
$$\lim_{p\to\infty}\sum_{x\in \calO^t\setminus \{0\}}(r\beta)^{n^2mt}f(r\beta i(x))=V^{-1}\int_{\IR^{n^2mt}}f(x)dx$$
since $\beta\to 0^+$ as $p$ becomes large. Moreover, by the second part of the non-degeneracy assumption the ratio $\frac{\vert U_p\vert}{\vert R_p\vert^t}\to 1$, so that $\frac{p^{d(k)}}{\vert U_p \vert}$ approaches $\beta^{n^2mt}$ as $p\to \infty$. 
Finally, switching the limit in $p$ and summation in $r$ is allowed by dominated convergence, as $f$ decays rapidly.  
\end{proof}
We now turn to examples of such a setup. \begin{example}
 In the case where $n=1$ and $\calO$ is just the ring of integers $\calO_K$ of a number field $K$, we can for example take $\phi_p$ to be the reduction map modulo a prime $\ip\vert p$ of $\calO_K$ for (the infinitely many) primes $p$ that split completely in $\calO_K$. In this case, we get $\calR_p=\IF_p$ and take $U_p$ to be the complement in $\calR^t$ of $(\IF_p\setminus\IF_p^\times)^t$. With the usual definition of $k$-dimensional codes as free rank $k$ $\IF_p$-submodules of $\IF_p^t$, the balancedness condition is satisfied (see e.g., \cite[Lemma 1]{Loeliger97averagingbounds}), and moreover the non-degeneracy result is straightforward in this case: nonzero elements $x\in\ip \calO_K$ have algebraic norm
    $$N_{K/\IQ}(x)\in p\IZ\setminus \{0\},$$ 
    but the norm is just the product of the $n$ embedded coordinates of $i(x)$, so by the geometric/arithmetic mean inequality we obtain
    $$\|i(x)\|\gg  p^{1/n}$$
    and non-degeneracy is satisfied. The condition on $k$ in Theorem \ref{thm:mainaverage} simply becomes $k>0$ and this recovers the construction of \cite{CampelloRandom} in the number field case. 
    \end{example}
   \begin{example}\label{example:firstresult}
    Let now $n>1$ so that we are in the non-commutative situation. The most straightforward generalization to division algebras $A$ with center $K$ of the number field results is the following: consider primes $\ip\vert p$ of $\calO_K$ that split $A$. Let $\IF_q\cong \calO_K/\ip$ denote the residue field. We then have an infinite family of ring isomorphisms 
    $$\phi_p: \calO/\ip\calO\to M_n(\IF_q)$$ 
    for a maximal order $\calO$ in $A$. We take $t$ copies of these to build our reduction map. In order to make sure the multiplicative structure is preserved as well, it makes sense to set $\calR_p:=M_n(\IF_q)$ and to consider codes which are free $\calR_p$-submodules of $\calR_p^t$ of rank $k$. It also seems natural to define $U_p$ to be the complement in $\calR_p^t$ of 
        $$(\calR_p^t\setminus U_p):=((M_n(\IF_q)\setminus \GL_n(\IF_q))^t.$$
        It then follows that the $U$-balancedness condition is again satisfied (see the proof of \cite[Lemma 2]{CampelloRandom}).
    We get a non-degeneracy result from Lemma \ref{lemma:lowerboundbadpoints}: indeed we obtain for a nonzero vector $a\in \calO^t$ that maps to $\calR_p^t\setminus U_p$ the lower bound 
    \begin{equation}\label{eq:lowerboundnorm}
        \|i(a)\|\gg  q^{\frac{1}{nm}}
    \end{equation}
    Moreover, as $\vert \GL_n(\IF_q)\vert=\prod_{i=0}^{n-1}(q^n-q^i)$ we also have that $\frac{\vert U_p\vert}{\vert \calR_p\vert^t}\to 1$ as $p\to \infty$. Since the largest possible $\IF_q$-dimension of a code is $n^2(t-1)$, tracing through the definitions we have non-degeneracy exactly when $t>n$, in which case via Theorem \ref{thm:mainaverage} we obtain the averaging result by pulling back codes which are free $M_n(\IF_q)$-modules of rank $k$ for $t>k>t-t/n$. In particular, when $n=2$ we recover in the special case of the quaternion algebra $A=\left(\frac{-1,-1}{\IQ}\right)$ the results of \cite[Theorem 3]{CampelloRandom} for the Lipschitz integers as well as for the maximal order of Hurwitz integers. We note that we recover the condition $k>t/2$ appearing in \cite[Theorem 3]{CampelloRandom}. So this does give a generalization to arbitrary $\IQ$-division algebras when $t>n$, however it would be more convenient to obtain results that work for arbitrary $t\geq 2$.

\end{example}

In order to obtain averaging results that do not require the condition $t>n$, one would have to obtain a stronger non-degeneracy lower bound on the norm or relax the condition that the codes be free $\calR_p=M_n(\IF_q)$-modules. Since the norm-trace inequality and the lower bound are sharp, we focus on the latter. However, when relaxing the definition of codes, one has to be careful to preserve the $U$-balancedness condition and to preserve enough multiplicative structure so that the units $\calO^\times$ act on the lattices and we obtain improved bounds on the packing density. This is achieved by taking as $k$-dimensional codes $k$ copies of the simple left $M_n(\IF_q)$-module $V=\IF_q^n$ for $n\leq k< tn$. This carries enough structure through to the lattices in $\mathbb{L}_p$ for the intended applications. Moreover, the $U$-balancedness result is a special case of the following, which we formulate in representation-theoretic terms: 
\begin{lemma}
   Let $k$ be a field. Let $R$ be a f.d. semisimple $k$-algebra and $V$ be a simple (left) $R$-module of finite dimension over $k$. Fix integers $n_1 \le n_2 \le n_3$. Consider $V^{\oplus n_3}$ as an $R$-module and consider the sets
   \begin{align*}
     U & = \{ v \in V^{\oplus n_3}\ | \ Rv \simeq  V^{\oplus n_1}\}, \\
     \mathcal{C}_{n_2,n_3} & = \{  C \subseteq V^{\oplus n_3} \ | \ C \text{ is an $R$-submodule, }C \simeq V^{\oplus n_2}\}.
   \end{align*}
Assuming that $U$ is non-empty, for each $u \in U$ there is a  bijection 
\begin{align*}
  \{ C \in \mathcal{C}_{n_2,n_3} \ | \ u \in C\} \leftrightarrow \mathcal{C}_{n_2-n_1,n_3-n_1}
\end{align*}
 \end{lemma}
 \begin{proof}
   Observe that for any $u \in U$
   \begin{align*}
     u \in C \Leftrightarrow &\  R u \subseteq C  \subseteq V^{\oplus n_3}\\
     \Leftrightarrow & \ \frac{C}{ Ru}\subseteq \frac{ V^{\oplus n_3} }{ Ru } \simeq V^{\oplus (n_3 - n_1)}.
   \end{align*}
   Hence, if we identify $V^{\oplus n_3}/Ru$ with $V^{\oplus (n_3 - n_1)}$, the proposed bijection is simply $C \mapsto C/Ru$.
 \end{proof}
 \begin{corollary}\label{cor:balancedness}
   In the previous lemma, if $k$ is a finite field, then the number $\#\{ C \in \mathcal{C}_{n_2,n_3} \ | \ u \in C\}$ is independent of $u$.
 \end{corollary}

 Setting $R= M_n(\mathbb{F}_q)$, $V=\mathbb{F}_q^{n}$, $n_3 = nt$, $n_2 = k$ and $n_1 = n$, we deduce our $U$-balancedness condition. We summarize our results in this case as a consequence of Theorem \ref{thm:mainaverage} in easily accessible form: 
 
 \begin{theorem} \label{thm:specificaverage}
 Let $A$ be a $\IQ$-division algebra whose center is a number field $K$. Let $n=\sqrt{[A:K]}$ and $m=[K:\IQ]$. Let $\calO$ be an order in $A$ and for an integer $t\geq 2$ we consider an infinite family of surjective reduction maps 
 $$\phi_p:\calO^t\to M_n(\IF_q)^t$$
 as given in each coordinate by Lemma \ref{lemma:primesexistence}. Let $i:A^t\to (A\otimes_\IQ\IR)^t$ denote the coordinate-wise embedding and let $f:\IR^{n^2mt}\to \IR$ be a semi-admissible function. For a fixed $n\leq k < nt$, set 
  \begin{align*}
    \mathcal{C}_{k,p} & = \{ C \subseteq M_n(\mathbb{F}_q)^{\oplus t } \ | \ C \text{ is a $M_n(\mathbb{F}_q)$-submodule isomorphic to } (\mathbb{F}_q^{n})^{\oplus k}\}\\
    \mathbb{L}_{k,p} & = \{ \beta_{p} \phi_p^{-1}(C) \ | \ C \in \mathcal{C}_{k,p}\},
  \end{align*}
 where the constant $\beta_p$ normalizing the covolume of lattices in $\mathbb{L}_{k,p}$ to $V:=\Vol(\calO^t)$ is given by $\beta_p=q^{\frac{nk-n^2t}{n^{2}mt}}$. Then if $(n-1)t<k<nt$, we have that
 $$\lim_{p\to\infty}\mathbb{E}_{\mathbb{L}_{k,p}}\left(\sum_{x\in (\beta_p\phi_p^{-1}(C))'}f(i(x))\right)\leq (\zeta(n^2mt)\cdot V) ^{-1}\cdot \int_{\IR^{n^2mt}}f(x)dx,$$
 where the limit is taken over primes in the family and $(\beta_p\phi_p^{-1}(C))'$ denotes the primitive vectors in $\beta_p\phi_p^{-1}(C)$.
 \end{theorem}
  \begin{proof}
  This readily follows from our discussions above as a special case of Theorem \ref{thm:mainaverage}: indeed we set
  $$ U_{p} = \{ v \in M_n(\mathbb{F}_q)^{\oplus t} \ | \ \dim_{\mathbb{F}_q}\left( M_n(\mathbb{F}_q) v  \right) = n^{2} \}.$$
  Moreover, we take $R_p=M_n(\IF_q)$ and the U-balancedness condition is satisfied by Corollary \ref{cor:balancedness}. Finally, the non-degeneracy condition is again satisfied as in Example \ref{example:firstresult}, since if $a \in \mathcal{O}^{\oplus t} \setminus \{ 0\}$ is such that $\phi_p(a) \notin U_{p}$, then $a$ has to have one coordinate which is non-trivial and is a non-invertible matrix modulo $p$. Thus we obtain the bound in equation \eqref{eq:lowerboundnorm} and obtain non-degeneracy as before. Finally, the condition on $d(k)$ in Theorem \ref{thm:mainaverage} then just becomes $tn-t<k<tn$. 
  \end{proof}
  In particular, we obtain in this way a valid result as soon as $t\geq 2$.
\section{Improved bounds}\label{sec:four}
Keeping the notation from the previous section, we now show how to leverage the extra symmetries under finite groups $G_0\subset \calO^\times$ of the lattices obtained in $\mathbb{L}_p$ in order to obtain sphere packings of density exceeding the Minkowski--Hlawka bound. We first present a result based on the approach in \cite[Corollary 1]{CampelloRandom} which in turn is inspired by Vance \cite{VanceImprovedBounds} and Rogers' \cite{RogersExistence(Annals)} work. \par

For a central $K$-division algebra $A$, we recall that $A_{\IR}$ denotes the real vector space $A\otimes_{\IQ}\IR$ of dimension $n^2m$. We also recall that the space $A_{\mathbb{R}}^{t}$ is endowed with a norm coming from the quadratic form as defined in Equation \ref{eq:eva_norm} for some positive definite and symmetric $a \in A_{\mathbb{R}}$. Such a norm will be chosen and fixed permanently in Lemma \ref{lemma:unitaction}.

Given a lattice $\Lambda\subset A_{\IR}^t$ which is an $\calO$-module and such a choice of norm we however first define the \emph{k-th $A$-minimum} $\min_k(\Lambda)$ to be the smallest $r$ such that the closed ball $\mathbb{B}_{A_{\IR}}(r)$ of radius $r$ contains $k$ $A_{\IR}$-linearly independent lattice vectors (under the left $A_{\mathbb{R}}$-action on $A_{\mathbb{R}}^{t}$). \par

In particular, $\min_1(\Lambda)$ is the shortest vector length $\lambda_1(\Lambda)$ in $\Lambda$. We begin by remarking that a lemma of Minkowski \cite{minkowski1910geometrie} which was extended by Vance \cite[Theorem 2.2]{VanceImprovedBounds} holds even more generally: 

\begin{prop}\label{prop:productminima}
Let $t\geq 2$ and $\Lambda$ denote an $\calO$-lattice in $A_{\IR}^t$. Then $\Lambda$ contains a left $A_{\IR}$-module basis $\{v_1,\ldots,v_t\}$ such that $\|v_i\|=\min_i(\Lambda)$. Moreover, if $\Vol(\Lambda)=1$, there exists an $\calO$-lattice ${\Lambda'}$ of covolume one in $A_{\IR}^t$ such that 
$$\lambda_1({\Lambda'})=\left(\prod_{i=1}^t\Min_i(\Lambda)\right)^{1/t}.$$
\end{prop}
\begin{proof}
Essentially the same proof goes through as in \cite[Theorem 2.2]{VanceImprovedBounds} (replacing $4$ by the appropriate dimension $mn^2$). 

To proceed, we select 
  \begin{align*}
    v_1 & =    \argmin_{v \in \Lambda \setminus \{ 0\}}\  \|i(v)\| \\
    v_2 &= \argmin_{v \in \Lambda \setminus A v_1} \|i(v)\|\\
    v_3 &=  \argmin_{v \in \Lambda \setminus (A v_1 + A v_2)} \|i(v)\|\\
    v_4 & =   \argmin_{v \in \Lambda \setminus (A v_1 + A v_2 + A v_3)} \|i(v)\| \\
    & \vdots 
  \end{align*}
  and we can argue inductively that they are linearly independent with respect to $A_{\mathbb{R}}$-action and satisfy $\|i(v_i)\| = \min_{i}(\Lambda)$.

  We now generate vectors $x_1,x_2,\cdots,x_k$ using a Gram-Schmidt process (see Appendix \ref{se:gs_process}) on $v_1,v_2,\cdots,v_k$. 
  Since $\{ v_i\}_{i=1}^{k}$ is free with respect to left-$A_{\mathbb{R}}$ action, we get that $x_1,x_2,\cdots,x_k$ freely generate $A_{\mathbb{R}}^{k}$ as a left-$A_{\mathbb{R}}$ module.
  
  Now consider the $\mathbb{R}$-linear map given by 
  \begin{align*}
  T: y \mapsto \frac{y_i}{\lambda_i} x_1 + \frac{y_2}{ \lambda_2} x_2 + \cdots \frac{y_k}{ \lambda_k} x_k , \\ \text{ for }y = y_1x_1+y_2x_2 + \cdots +y_kx_k \in A_{\mathbb{R}}^{k}.
  \end{align*}
  
  Define a lattice $\Lambda'$ by:
\begin{align*}
  \Lambda' =  (\lambda_1 \lambda_2 \ldots \lambda_k)^{ {1}/{k}}T(\Lambda).
\end{align*}
  We observe that $\det(\Lambda') = \det(\Lambda)$. For any $y' \in \Lambda' \setminus \{ 0\}$, we can now find $y = y_1 x_1 + \cdots + y_k x_k \in \Lambda$ such that $y' = (\lambda_1\cdots \lambda_k)^{1/k} T ( y)$. Furthermore, there must be a smallest $i_0 \ge 1$ such that $y_{i_0} \neq 0$ and $y_{i_0+1} = y_{i_0+2} = \cdots = y_{k} = 0$. Then $y \in \Lambda \cap ( A_{\mathbb{R}} v_1 + A_{\mathbb{R}} v_2 + \cdots + A_{\mathbb{R}} v_{i_0})$ and  $y\not\in\Lambda \cap ( A_{\mathbb{R}} v_1 + \cdots + A_{\mathbb{R}} v_{i_0 - 1} )$. Therefore $\|i(y)\| \ge \lambda_{i_0}$. This implies that 
\begin{align*}
  \|i(y')\|^{2} = \left( \lambda_1 \lambda_2 \ldots \lambda_k \right)^{2/k} \sum_{i=1}^{i_0} \left| \frac{i(y_i)}{\lambda_i}\right|^{2}   \ge \left( \lambda_1 \lambda_2 \ldots \lambda_k \right)^{2/k} \frac{1}{\lambda_{i_0}^{2}} \sum_{i=1}^{k}\left| {i(y_1)}\right|^{2} \ge (\lambda_1 \ldots \lambda_k)^{2/k},
\end{align*}
using orthogonality as in Appendix \ref{se:gs_process}. This lower bound is tight, since we can set $y_1=1_{A_{\mathbb{R}}}$ and $y_i=0$ for $i \ge 2$.
\end{proof}
\begin{remark}\label{rem:effectiveMinkowski}
  This version of Minkowski's lemma given above is effective in the sense that for an explicit lattice $\Lambda$,  the lattice $\Lambda'$ can also be computed algorithmically. The only important step here is to find the successive minima vectors $v_i$, and then $\Lambda'$ is easily seen to be computable.
  
  When $A=\mathbb{Q}$ (so $A_\IR = \IR$), finding the $v_i$ can be achieved by the so-called {\bf SMP} algorithm, which will have an exponential running time of $O(2^{2t})$. Details can be found in \cite{micciancio2012complexity}. The division algebra case requires only slight modifications of the algorithm and should have a running time of $O(2^{2mn^2t})$.
\end{remark}

We also record the lemma: 
\begin{lemma}\label{lemma:unitaction}
  Let $\calO$ denote an order in a $K$-division algebra $A$ and let $G_0\subset \calO^*$ denote a finite group. Then $G_0$ acts on vectors in any $\calO$-lattice $\Lambda\in \mathbb{L}_p$ obtained as in the construction of Theorem \ref{thm:mainaverage}. Furthermore, we may choose a symmetric positive definite element $a \in A_\IR$ such that for all such $\Lambda$ the induced norm satisfies
  $$\|i(x)\|^2=\sum_{i=1}^{t} \T(x_{i}^{*}ax_i) = \|i(g\cdot x)\|^2 \text{ }\forall g\in G_0, x\in \Lambda.$$
\end{lemma}
\begin{proof}
The lattices obtained in $\mathbb{L}_p$ via our construction are easily seen to be preserved under the $\calO$-action when the morphisms $\phi_p$ preserve the multiplicative structure and the codes in $\calC$ we are pulling back are $\phi_p(\calO)$-modules. Therefore the units act as well. \par
For the second part, we may set $$a=\sum_{ g \in G_0} g^* g.$$ One can easily check that the induced quadratic form then has the required $G_0$-invariance.

\end{proof}
From now on, we may and will assume a norm as in Lemma \ref{lemma:unitaction} has been chosen on $A_\IR$.
Using the methods from \cite{RogersExistence(Annals),VanceImprovedBounds,CampelloRandom} we apply Theorem \ref{thm:mainaverage} to a specific function $f:\IR^{mn^2t}\to \IR$ in order to obtain improved bounds. 
\begin{theorem}\label{thm:improvedbounds}

Let $A$ be a central simple division algebra over a number field $K$. Let $\calO$ be an $\calO_K$-order in $A$. Let $n^2=[A:K]$, $m=[K:\IQ]$ and let $t\geq 2$ be a positive integer. Let $G_0$ be a fixed finite subgroup of $\calO^\times$. Then there exists a lattice $\Lambda$ in dimension $n^2mt$ achieving 
$$\Delta(\Lambda)\geq\frac{\vert G_0\vert\zeta(mn^2t)\cdot t}{2^{mn^2t}\cdot e(1-e^{-t})}.$$
Moreover, there exists for any $\varepsilon>0$ an $\calO$-lattice $\Lambda_\varepsilon$ in dimension $n^2mt$ achieving $$\Delta(\Tilde{\Lambda}_\varepsilon)\geq(1-\varepsilon)\cdot \frac{\vert G_0\vert\zeta(mn^2t)\cdot t}{2^{mn^2t}\cdot e(1-e^{-t})}$$
which can be constructed. Indeed, $\Tilde{\Lambda}_\varepsilon$ is obtained by applying Proposition \ref{prop:productminima} to a suitable sublattice of $\calO^t$ obtained as a pre-image via reduction modulo primes $\ip$ of $\calO_K$ of large enough norm of a code isomorphic to $k$ copies of simple left $\calO/\ip \calO$-modules for $nt-t<k<nt$ as in Theorem \ref{thm:specificaverage}.

\end{theorem}
\begin{proof}

We define $f$ to be the radial function $f_r$ of bounded support given by
$$f_r(y)=\begin{cases}\frac{1}{mn^2}& \text{ if }0\leq \|y\|<re^{(1-t)/mn^2t} \\
\frac{1}{mn^2t}-\log (\frac{\|y\|}{r})&\text{ if }re^{(1-t)/mn^2t}\leq \|y\|\leq re^{1/mn^2t}\\
0& \text{ else }
\end{cases}$$
This function is indeed semi-admissible and we have that 
$$\int_{\IR^{mn^2t}}f_r(y)dy=V_{mn^2t}\cdot r^{mn^2t}\cdot \frac{e(1-e^{-t})}{mn^2t},$$
where $V_{mn^2t}$ denotes the volume of the unit ball in $mn^2t$-dimensional Euclidean space. 
For a small $0<\varepsilon<1$ we may find $r\geq 0$ so that 
$$V_{mn^2t}\cdot r^{mn^2t}\cdot \frac{e(1-e^{-t})}{mn^2t}=(1-\varepsilon)\cdot \frac{\vert G_0\vert \vert \Vol(\calO^t)\vert \zeta(mn^2t)}{mn^2}.$$
Taking $\mathbb{L}_p$ and $k$ satisfying the assumptions of Theorem \ref{thm:mainaverage}, we may therefore for $p$ large enough find a lattice $\Lambda=\beta\cdot i(\Lambda_0)\in \mathbb{L}_p$ of volume $\Vol(\calO^t)$ such that 
$$\sum_{y\in \Lambda'}f_r(y)\leq (1-\varepsilon)\frac{\vert G_0\vert}{mn^2}<\frac{\vert G_0\vert}{mn^2}.$$
We now use the fact that the units of finite order $G_0<\calO^\times$ act freely on primitive vectors of $\Lambda$ and that $\| i(gv)\|=\|i(v)\|$ for $g\in G_0$ for our choice of norm (see Lemma \ref{lemma:unitaction}). Indeed, letting $\{v_1,\ldots,v_t\}$ be linearly independent vectors achieving the $A$-minima $\|\beta\cdot i(v_j)\|=\min_j(\Lambda)$ as guaranteed by Proposition \ref{prop:productminima}, we then have that 
$$\sum_{y\in \Lambda'}f_r(y)\geq \sum_{j=1}^t\sum_{g\in G_0}f_r(\beta\cdot  i(gv_j))=\vert G_0\vert\sum_{j=1}^t f_r(\beta\cdot i(v_j)). $$
In other words, $\sum_{j=1}^t f_r(\beta\cdot i(v_j))<1/(mn^2)$ so that by definition of $f_r$ we must have 
\begin{equation}\label{eq:minimaboundone}
    \min_j(\Lambda)\geq r e^{(1-t)/(mn^2t)} \text{ for all }j.
\end{equation}
Moreover, it must then be by definition of $f_r$ that 
\begin{equation}\label{eq:minimaboundtwo}
    \sum_{j=1}^t \log \left(\frac{\min_j(\Lambda)}{r}\right)>0
\end{equation}
and hence 
$$\left(\prod_{j=1}^t \min_j(\Lambda)\right)^{1/t}>r.$$
From proposition \ref{prop:productminima} we deduce the (constructive) existence of a lattice $\Tilde{\Lambda}$ with volume equal to $\Vol(\Lambda)$ and shortest vector length $\lambda_1(\Tilde{\Lambda})>r$. We thus obtain for all such $\varepsilon$ the existence of a lattice $\Tilde{\Lambda}_\varepsilon$ of volume $\Vol(\calO^t)$ and packing density 
$$\Delta(\Tilde{\Lambda}_\varepsilon)\geq(1-\varepsilon)\cdot \frac{\vert G_0\vert\zeta(mn^2t)\cdot t}{2^{mn^2t}\cdot e(1-e^{-t})}.$$
Letting $\varepsilon\to 0$, it thus also follows by Mahler compactness that the sequence $\Tilde{\Lambda}_\varepsilon$ has a converging subsequence (in the quotient topology on $\GL_{mn^2t}(\IR)/\GL_{mn^2t}(\IZ)$). Since the packing density is a continuous function with respect to this topology, we also get the existence of a lattice with density $\Delta_t\geq\frac{\vert G_0\vert\zeta(mn^2t)\cdot t}{2^{mn^2t}\cdot e(1-e^{-t})}$.
\end{proof}
\begin{remarks}
\begin{enumerate}
    \item 
First note that the zeta factor quickly approaches one as $n^2mt$ increases and thus can be ignored for the purpose of giving asymptotic bounds for large dimensions.
\item
The lower bounds on the density in Theorem \ref{thm:improvedbounds} have the advantage of producing a factor $t$ in the numerator for lattices constructed from $\calO^t$. Via a simpler approach, taking $f$ to be the indicator function of a ball one finds an $\calO$- lattice $\Lambda$ 
which outperforms the bound above (slightly) only when $t=2$. We record this below. 
\end{enumerate}
\end{remarks}
\begin{prop}\label{prop:simpleimprovedbounds}
With the notations of Theorem \ref{thm:improvedbounds}, there exists a $n^2mt$-dimensional sub-lattice $\Lambda_{\varepsilon}\subset \calO^t$ with packing density
$$\Delta(\Lambda_{\varepsilon})\geq (1-\varepsilon)\cdot\frac{\vert G_0\vert\zeta(mn^2t)}{2^{mn^2t}}$$
in the set of scaled pre-images of codes $\mathbb{L}_p$ for $p$ large enough. Moreover, there exists a $n^2mt$-dimensional lattice $\Lambda$ with density $\Delta(\Lambda)\geq \frac{\vert G_0\vert\zeta(mn^2t)}{2^{mn^2t}}$ for all $t\geq 2$. 
\end{prop}
\begin{proof}
Take $f$ to be the indicator function of a ball of radius $r$ and let $r$ be chosen so that $\Vol(\IB(r))=(1-\varepsilon)\vert G_0\vert\zeta(mn^2t)\Vol(\calO^t)$. Applying Theorem \ref{thm:specificaverage}, there must be for large enough $p$ a lattice $\Lambda_\varepsilon$ in $\mathbb{L}_p$ (with the notations of the theorem) such that 
\begin{equation}\label{eq:orbittrick}\vert \IB(r)\cap\Lambda_\varepsilon'\vert\leq (1-\varepsilon)\vert G_0\vert.\end{equation}
Having arranged for $f$ to be $G_0$-invariant, the left hand side of \eqref{eq:orbittrick} must be a multiple of $\vert G_0\vert$ and thus is forced to equal $0$. This gives a lower bound on the shortest vector leading to the desired packing density for $\Lambda_\varepsilon$, while the second statement follows again by Mahler compactness.
\end{proof}
Taking $t=2$ is often the most advantageous in view of optimizing the packing density in relation to the dimension of the lattice. Nevertheless, having improved bounds for arbitrary $t\geq 2$ should be useful.

\subsection{Classification of finite subgroups of $\calO^\times$ and bounds.}
In order to examine the density of lattice packings that can be achieved via this method, it is necessary to understand which finite groups $G_0<\calO^{\times}$ can occur for $\IQ$-division algebras. This classification was completely carried through by Amitsur \cite{Amitsur1955FiniteSO}. As outlined in more detail in \cite[2.2--3]{gargava2021lattice}, we summarize some cases that lead to new, dense packings. \par
To that end, we recall that recently dense packings were found in special cases of our construction by Venkatesh \cite{VenkateshBounds} when $A=K=\IQ(\zeta_m)$ and by Vance \cite{VanceImprovedBounds} when $A=\left(\frac{-1,-1}{\IQ}\right)$ by exploiting that the respective $\calO$-lattices are invariant under $\IZ/m\IZ$ and the binary tetrahedral group $\mathfrak{T}^*\cong \SL_2(\IF_3)$ of order $24$, respectively. The most spectacular lattice packing densities in the cyclotomic case then occur when maximizing the ratio $\frac{\vert \IZ/m\IZ\vert}{[\IQ(\zeta_m):\IQ]}=\frac{m}{\varphi(m)}$ whereas in the Hurwitz lattice case the improved bounds are obtained in dimensions $4t$. The first result is that in the more general context of division algebras we may in some sense combine these two improvements: 
\begin{prop}
Assume $m$ is a positive integer such that $2$ has odd order modulo $m$. Then the algebra $\IQ(\zeta_m)\otimes_\IQ\left(\frac{-1,-1}{\IQ}\right)$ is a division algebra with center $\IQ(\zeta_m)$ and has a maximal $\IZ[\zeta_m]$-order $\calO$ with subgroup $\mathfrak{T}^*\times \IZ/m\IZ\subset\calO^\times$. 
\end{prop}
\begin{proof}
See \cite[Theorems 6a, 7]{Amitsur1955FiniteSO}. 
\end{proof}
In particular, we obtain for $m$ satisfying the parity condition above the existence of lattices $\Lambda_m$ in dimension $8\varphi(m)$ satisfying: 
$$\Delta(\Lambda)\geq \frac{24m\zeta(8\varphi(m))}{2^{8\varphi(m)}}>\frac{24m}{2^{8\varphi(m)}}. $$
via lifting codes as in Theorem \ref{thm:specificaverage} and Proposition \ref{prop:simpleimprovedbounds}. 
By maximizing the ratio $m/\varphi(m)$ under the additional parity condition, we arrive at: 
\begin{prop}\label{prop:cycloHurwitz}
Using the construction above and letting 
$$m_k=\prod_{\substack{p\leq k \text{ prime}\\2\nmid\ord_2p}}p,$$
we obtain lattice packings in dimension $8\varphi(m_k)$ of density 
\begin{equation}\label{eq:cycloquatdensity}
    \Delta \geq C(\log\log \varphi(m_k))^{7/24}\cdot \frac{24\cdot\varphi(m_k)}{2^{8\varphi(m_k)}}
    \end{equation}
    
for some fixed constant $C>0$. Moreover, for any $C<1$ the bound in \eqref{eq:cycloquatdensity} is valid for $m_k$ large enough. 
\end{prop}
\begin{proof}
This is \cite[Theorem 30]{gargava2021lattice}. 
\end{proof}
Moreover, there are several other ways to obtain dense packings in new dimensions via the division algebra approach. We refer the reader to the discussion in \cite{gargava2021lattice} but simply restate \cite[Prop 64]{gargava2021lattice}:
\begin{prop}\label{prop:loglogimprovement}
There exists an infinite sequence of dimensions $\{d_n\}$ in which a packing density 
$$\Delta_{d_n}\geq \frac{1}{2}\log\log d_n \frac{d_n}{2^{d_n}}$$
is achieved and the lattice achieving this packing density is invariant under the action of a non-commutative group.
\end{prop}
\begin{proof}
  See \cite[Prop 31]{gargava2021lattice}. The division algebras in question are $\left( \frac{-1,-1}{\IQ[\zeta_m + \zeta_m^{-1}]}\right)$ for $m$ a product of primes maximizing $m/\varphi(m)$ under suitable conditions. The extra symmetries here come from the action of the dihedral group with $2m$ elements on lattices in dimension $4\varphi(m)$.
\end{proof}
One of the main advantages in obtaining such lattices via lifts of codes is that, at least in theory, such lattices can be explicitly found by searching a finite set of parameters as opposed to, say, the averaging results in \cite{gargava2021lattice}. We conclude by a discussion of such effectivity questions.
\section{Notes on effectivity}\label{sec:five}
Our results such as Theorem \ref{thm:specificaverage} imply that dense lattices in dimension $mn^2t$ can be found among pre-images of codes in characteristic $p$ as $p\to\infty$. In this last section we show how large it suffices to take $p$ in order to guarantee a lattice of packing density greater than $(1-\varepsilon)\frac{\vert G_0\vert}{2^{mn^2t}} $ is found, with $G_0<\calO^*$ designating the units of finite order in $\calO$.\par 
\subsection{Varying the division ring}
We first focus on the case of $t=2$ in Theorem \ref{thm:improvedbounds} when in fact the better bounds are obtained by taking the simpler indicator function $f=\mathds{1}_{\IB(r)}$ of a ball of appropriate radius as in Proposition \ref{prop:simpleimprovedbounds}. \par
\begin{theorem}\label{thm:effective}
Let $A$ denote central simple division $K$-algebras for number fields $K$ and denote $[A:K]=n^2$ and $[K:\IQ]=m$. Let $\calO$ denote a maximal order in such $A$. Fix $0< \varepsilon< 1$. Assume the prime $\ip\vert p$ in $\calO_K$ is chosen large enough with respect to $m,n$ so that the size of the residue field $\vert \calO_K/\ip\vert =q$  satisfies: \begin{enumerate}
    \item we have as $m,n$ increase the relation:
$$(n^2m)^{2}\Vol(\calO)^{2/(mn^2)}\vert G_0\vert^{-1/(mn^2)}=o(q^{1/mn}),$$
\item the ratio $\frac{\vert M_n(\IF_q)\vert ^2}{\vert M_n(\IF_q)\vert ^2-\vert M_n(\IF_q)\setminus\GL_n(\IF_q)\vert^2}<(1+\varepsilon/3)$.
\end{enumerate}  
Then there exists an effective constant $C_\varepsilon>0$ such that in dimension $2n^2m>C_\varepsilon$ there exists a lattice $\Lambda\in \mathbb{L}_p$ with packing density 
$$\Delta(\Lambda)\geq (1-\varepsilon)\frac{\vert G_0\vert }{2^{2n^2m}}.$$
Here $\mathbb{L}_p$ denotes the set of scaled preimages of generalized codes of $\IF_q$-dimension $2n^2-n$ via the reduction map $\phi_p:\calO^2\to(\calO/\ip\calO)^2$ as in Theorem \ref{thm:specificaverage}. 
\end{theorem}
\begin{proof}

Tracing through the proof of Theorem \ref{thm:specificaverage} for $t=2$ and $k=2n-1$ (the only sensible choice), we find that the term 
$$\sum_{x\in (\phi_p^{-1}(C))', \phi_p(x)\in (\calR_p^2\setminus U_p)}f(\beta i(x))$$
is trivial for $f=\mathds{1}_{\IB(r)}$ and some $C \in \mathcal{C}_{k,p}$ as soon as 
\begin{equation}
    r< \left( \N(a)^{1/2n^2m}\cdot n\sqrt{m} \right)  q^{\frac{1}{2nm}},
\end{equation}
via Lemma \ref{lemma:lowerboundbadpoints} and \eqref{eq:lowerboundnorminthm}, where $a=\sum_{g\in G_0}g^*g$. Since $\N(g^*g)=1$ for $g\in G_0$, it is easy to give a uniform lower bound $\N(a)^{1/2n^2m}\geq 1$ or even $\N(a)^{1/2n^2m}\geq \sqrt{|G_0|}$ by the Minkowski determinant inequality. In particular, it suffices to ensure the parameter $r$ satisfies
\begin{equation}\label{eq:effectivepush}
    r< \left(  n\sqrt{m} \right)  q^{\frac{1}{2nm}}.
\end{equation}

The expected value for the remaining terms for fixed characteristic $p$ can then be seen by balancedness to be bounded by: 
\begin{equation}
    \mathbb{E}\leq \frac{q^{n(2n-1)}}{q^{2n^2}-(q^{n^2}-\prod_{i=0}^{n-1}(q^n-q^i))^2}\cdot \sum_{x\in (\calO^2)'} \mathds{1}_{\IB(r)}(\beta_p i(x))
\end{equation}
Now by a classical geometry of numbers result (see \cite[Lemma 4]{CampelloRandom} or \cite[Lemma 3 (2)]{MoustrouCodes}) we can bound 
\begin{equation}
    \sum_{x\in (\calO^2)'} \mathds{1}_{\IB(r)}(\beta_p i(x))\leq (r+\beta_p\tau(\calO^2))^{2n^2m}\cdot \frac{V_{2n^2m}}{\beta_p^{2n^2m}\Vol(\calO^2)},
\end{equation}
where $\tau(\calO^2)$ denotes the packing radius of $\calO^2$ and $V_d$ denotes the volume of the $d$-dimensional unit ball. 
Writing $S_n(q):=\frac{q^{n(2n-1)}}{q^{2n^2}-(q^{n^2}-\prod_{i=0}^{n-1}(q^n-q^i))^2}\geq \beta_p^{2n^2m}$ we arrive at:
\begin{equation}\label{eq:effectiveineq}
    \mathbb{E}\leq \frac{S_n(q)}{\beta_p^{2n^2m}} r^{2n^2m}\frac{V_{2n^2m}}{\Vol(\calO^2)}\cdot \left(1+\frac{\tau(\calO^2)\beta_p}{r}\right)^{2n^2m}
\end{equation}
Observe now that $\frac{S_n(q)}{\beta_p^{2n^2m}}=\frac{\vert M_n(\IF_q)\vert ^2}{\vert M_n(\IF_q)\vert ^2-\vert M_n(\IF_q)\setminus\GL_n(\IF_q)\vert^2}$, so that we can assume $q$ is large enough so that $\frac{S_n(q)}{\beta_p^{2n^2m}}<(1+\varepsilon/3)$. \par
Moreover, for a radius $r$ that yields the density bound, we should have that the volume of the ball of radius $r$ is around $\vert G_0\vert (1-\varepsilon)\Vol(\calO^2)$. By the Stirling formula, 
we may estimate as the dimension grows
$$r\sim \frac{n\sqrt{m}}{\sqrt{\pi e}}(\vert G_0\vert \Vol(\calO^2))^{1/2n^2m},$$
which under our assumptions on $q$ satisfies the inequality \eqref{eq:effectivepush} using a trivial bound like $\vert G_0\vert=o((n^2m)^{nm})$. It now suffices to show that under the parameters above, we can bound $\left(1+\frac{\tau(\calO^2)\beta_p}{r}\right)^{2n^2m}< (1+\varepsilon/3)$ for large enough dimension, since then we get from the inequality \eqref{eq:effectiveineq} the existence of a lattice in $\mathbb{L}_p$ with the desired lower bound on the pac king density. Recall that $\beta_p=q^{-1/2nm}$ with our parameters and we have from Lemma \ref{lemma:packingradius} that 
$$\tau(\calO^2)=\sqrt{2}\cdot \tau(\calO)\leq \Vol(\calO)^{2/n^2m}\cdot (n\sqrt{m}+6)/(\sqrt{2}\pi).$$
We thus have
\begin{equation*}
    \frac{\tau(\calO^2)\beta_p}{r}\lesssim \Vol(\calO)^{1/(mn^2)} \vert G_0\vert^{-1/(2mn^2)}
    q^{-1/(2mn)}
\end{equation*}
But under the assumptions of the theorem on $q$, the result now follows since as $mn^2$ goes to infinity the term $2mn^2\cdot \frac{\tau(\calO^2)\beta_p}{r}$ becomes arbitrarily small. Assuming $p$ and $q$ chosen large enough for each $n,m$ as in the assumptions of the theorem, we may thus view  $\left(1+\frac{\tau(\calO^2)\beta_p}{r}\right)^{2n^2m}$ as a function in $n,m$ which approaches $1$ for large enough dimension $n^2m$. This easily yields an effective constant $C_\varepsilon$ guaranteeing $\left(1+\frac{\tau(\calO^2)\beta_p}{r}\right)^{2n^2m}<(1+\varepsilon/3)$ for $n^2m>C_\varepsilon$. 
\end{proof}
We may then for instance apply this result to specific families of maximal orders in division rings of increasing $\IQ$-dimension. One may arrange for the size of the finite units $G_0$ to be known in this family via Amitsur's results (\cite{Amitsur1955FiniteSO}). Moreover, the computation of the volume $\Vol(\calO)$ reduces to a computation of $\sqrt{d(\calO/\IZ)}$, since the $\IZ$-discriminant $d(\calO/\IZ)$ can be defined as the ideal generated by $\{\det(\tr_{A/\IQ}x_ix_j)_{1\leq i,j\leq [A:\IQ]}\}$ for $x_i\in \calO$ a $\IZ$-basis. Equivalently it is the norm of the $\IQ$-algebra different $\N_{A/\IQ}(\mathfrak{D}(\calO/\IZ))$. We may then write (\cite[Ex 25.1]{reiner2003maximal}):
$$d(\calO/\IZ)=\N_{K/\IQ}(d(\calO/\calO_K))\cdot d(\calO_K/\IZ)^{n^2},$$
where $d(\calO/\calO_K)$ is just the regular discriminant of the central simple $K$-algebra $A$ and $d(\calO_K/\IZ)$ is the discriminant of the central number field. In particular, when there is some control of the ramification behavior of $A/K$ and we have some upper bounds for the discriminant of $K$, the conditions of Theorem \ref{thm:effective} become entirely explicit. 

\begin{example}
 When $K=\IQ(\zeta_m)$ is a cyclotomic field, one has that 
 \begin{equation}\label{eq:discriminanteffective}
 d(\calO_K/\IZ)=\frac{m^{\varphi(m)}}{\prod_{l\in \IP, l\mid m}l^{\varphi(m)/(l-1)}}.\end{equation}
 When $n=1$, by considering cyclotomic fields $\IQ(\zeta_m)$ Moustrou thus finds via a version of Theorem \ref{thm:effective} in this case effective dense lattices in dimensions $2\varphi(m)$ for large enough $m$ and shows a suitable $q$ can be found in time $O(m^3\log(m))^{\varphi(m)}$, see \cite[Theorem 1, Prop 3.1]{MoustrouCodes}. 
\end{example}
\begin{example}
Similarly, fixing $n=2$, varying $K=\IQ(\zeta_m)$ and considering the quaternion algebra over $K$ $$A=\left(\frac{-1,-1}{\IQ(\zeta_m)}\right)$$
when $2$ has odd order modulo $m$, we can use that $d(\calO_K/\IZ)\leq m^{\varphi(m)}$ and that the discriminant of the Hurwitz integers $d(\calH/\IZ)=2$ to obtain an effective version of Proposition \ref{prop:cycloHurwitz} via Theorem \ref{thm:effective}. We record this below. 
\end{example}

\begin{prop}\label{prop:cycloquaternionseffective}
  Let $m_k=\prod_{\substack{p\leq k \text{ prime}\\2\nmid\ord_2p}}p$ and set $n_k:=8\varphi(m_k)$. Then for any $\varepsilon>0$ there is an effective constant $c_\varepsilon$ such that for $k>c_\varepsilon$ a lattice $\Lambda$ in dimension $n_k$ with density 
  $$\Delta(\Lambda)\geq (1-\varepsilon)\frac{24\cdot m_k}{2^{n_k}}$$
  can be constructed in $e^{4.5\cdot n_k\log(n_k)(1+o(1))}$ binary operations. 
  This construction leads to the asymptotic density of 
  $$\Delta(\Lambda)\geq (1-e^{-n_k})\frac{3\cdot n_k(\log\log n_k)^{7/24}}{2^{n_k}}$$
  in dimension $n_k$. 
\end{prop}
\begin{proof}
We consider the quaternion algebras $A_k=\left(\frac{-1,-1}{\IQ(\zeta_{m_k})}\right)$ and exhibit a large enough residue field size $q$ so that the conditions of Theorem \ref{thm:effective} are satisfied. From the discriminant relation \eqref{eq:discriminanteffective} we obtain that the first condition amounts to 
$$(m_k\varphi(m_k)^2)^{2\varphi(m_k)}=o(q).$$
It is convenient to as in \cite[Prop 3.1.]{MoustrouCodes} search for large primes which split completely in $\IQ(\zeta_{m_k})$, and this happens when $p\equiv 1\mod m_k$. Using an effective version of the \v Cebotarev density theorem (see e.g., \cite{effectiveCebo}), one sees that an interval of size, say, $(1/m_k)(m_k)^{6\varphi(m_k)}$ around $m_k^{6\varphi(m_k)}$ must contain such a prime for large enough $m_k$. Such a prime then satisfies $(m_k\varphi(m_k)^2)^{2\varphi(m_k)}=o(p)$ for our choice of $m_k$. Finding a suitable prime $p$ can thus be done in at most around $e^{3/4 n\log n}$ steps. Moreover, for such primes, we have that 
$$\left\vert \frac{\vert M_2(\IF_p)\vert ^2}{\vert M_2(\IF_p)\vert ^2-\vert M_2(\IF_p)\setminus\GL_2(\IF_p)\vert^2}-1 \right\vert=o(e^{-n_k}),$$
which deals with the second condition of Theorem \ref{thm:effective}. The time estimate for enumerating the lattice family is then of $e^{4.5\cdot n_k\log(n_k)(1+o(1))}$ binary operations since the number of codes we consider in Theorem \ref{thm:effective} here amounts to $O(p^6)$. The costs of the remaining computations, such as computing the packing density of lattices, are also exponential in the dimension, but being of cost $2^{O(n_k)}$ do not contribute to the main term of the estimate. We thus obtain an effective version of the bounds in Proposition \ref{prop:cycloHurwitz}, as claimed.
\end{proof}
The density lower bounds above outperform the best known effective lower bounds on the density from cyclotomic fields up to dimensions around $1.98\cdot 10^{46}$ because of the improved constant coming from the size of the binary tetrahedral group ( see also \cite[Fig. 1]{gargava2021lattice}).
\begin{example}
Consider the case when $k=2$ and $m_k=7 \cdot 23=161$. A suitable prime $p$ is then for instance
$$p= (161 \cdot \varphi(161)^2)^{2 \varphi(161)} + 223147$$
of size about $10^{1072}$ and satisfying $p\equiv 1\mod 161$. Using this prime in Theorem \ref{thm:effective} yields a lattice $\Lambda$ in $8\cdot \varphi(161)=1056$ dimensions such that 
$$\Delta(\Lambda)\geq (1-\varepsilon)\frac{3864 }{2^{1056}},  $$
for $\varepsilon \le 10^{-1000}$ in about $10^{10^4}$ bit-operations. Even with the best computers, it is currently far out    of reach to enumerate a basis for such a lattice.  
Nevertheless, this gives a very explicit construction of such a random lattice that achieves a good packing. 
 \end{example}
 
Finally, we note that similarly an effective version of Proposition \ref{prop:loglogimprovement} can be obtained. It seems plausible for such results that as long as only the degree of the center of the division algebras over $\IQ$ is unbounded as in Proposition \ref{prop:cycloquaternionseffective}, the family size of $n$-dimensional lattices to be searched should be $e^{C\cdot n\log(n)(1+o(1))}$ for some constant $C>0$.  

\subsection{Varying the rank $t$}
Finally, we remark that one also obtains effective good asymptotic lattices from our constructions by fixing the division ring $A$ and maximal order $\calO$ and instead varying the rank of the $\calO$-lattices as in Vance's construction \cite{VanceImprovedBounds}. In particular, one obtains an effective version of Vance's construction which we record here. The general case is handled in the same way and is left to the reader. \par
\begin{prop}\label{prop:effectiveVance}
  For any $0<\varepsilon<1$, there exists a lattice in $\mathbb{H}^t$ which is a free rank $t$ module over the ring of Hurwitz integers $\calH$, whose geometric mean of the quaternionic minima satisfy 
  $$\left(\prod_{j=1}^t \min_j(\Lambda)\right)^{1/t}>r$$
  where $r$ is defined by $\Vol(\IB(r))=(1-\varepsilon)\frac{24t \Vol(\calH^t)\cdot \zeta(4t)}{e(1-e^{-t})}$,
  and which, provided the odd prime $p$ satisfies $t^2=o(p)$ and $t$ is large enough, lies in the set of (rescaled) lifts $$\mathbb{L}_p=\{p^{\frac{1-t}{2t}}\phi_p^{-1}(C):C\in \calC_{t+1}\},$$
  where $\phi_p:\calH^t\to(\calH/p\calH)^t\cong M_2(\IF_p)^t$ is the reduction map and $\calC_{t+1}$ is the set of left $M_2(\IF_p)$-submodules of $M_2(\IF_p)^t$ isomorphic to  $t+1$ copies of the simple left module $\IF_p^2$.
\end{prop}
\begin{proof}
Consider the proof of Theorem \ref{thm:improvedbounds}. Then for any $t\geq 2$ the support of the radial function $f_r(y)$ is contained in the ball of radius $re^{1/mn^2t}=re^{1/4t}$. Choose $r$ such that 
\begin{equation}\label{eq:roftsize}
    \Vol(\IB(r))=(1-\varepsilon)\frac{24t \Vol(\calH^t)\cdot \zeta(4t)}{e(1-e^{-t})}
\end{equation}
Via Stirling's formula, we get from \eqref{eq:roftsize} and the discriminant $d(\calH)=2$ the asymptotic 
\begin{equation}\label{eq:rapproxsize}
r\sim t^{1/(4t)+1/2}\cdot \frac{2^{5/8}}{\sqrt{\pi e}}.
\end{equation}
First consider any $t<k<2t$. Pulling back codes of $\IF_p$-dimension $2k$ as in Theorem \ref{thm:specificaverage}, in order to lift the averaging result we see that the support of $f$ has to be contained in the ball of radius $2p^{1/4}$, so that we arrive at the condition 
$$e^{(1+\ln(t)-2t)/(4t)}\cdot \frac{2^{-3/8}}{\sqrt{\pi} }\sqrt{t}<p^{1/4}.$$
Thus for $t\geq 2$ it in particular suffices to take $p\geq t^2$.
Note that here any odd prime $p$ is unramified and can be used in the construction. 
Inspecting the proof of Theorem \ref{thm:mainaverage} (and ignoring for simplicity the M\"obius inversion step since the zeta factor quicly approaches $1$ for large $t$), we have that 
\begin{align*}
   \mathbb{E}&\leq \frac{p^{d(k)}}{\vert U_p\vert\cdot \beta_p^{4t}}\cdot\sum_{x\in\calO^t\setminus \{0\}}\beta_p^{4t}f_r(\beta_p i(x))\\
    &=\frac{p^{4t}}{p^{4t}-(p^3+p^2-p)^t}\cdot\sum_{x\in\calO^t\setminus \{0\}}\beta_p^{4t}f_r(\beta_p i(x))
\end{align*}
as $\beta_p=p^{k/2t-1}$, and it therefore remains to bound the difference 
\begin{equation}
    \Delta(p,t)=\left |\beta_p^{4t}\cdot\sum_{x\in\calO^t\setminus \{0\}}f_r(\beta_pi(x))-V^{-1}\int_{\IR^{4t}}f_r(x)dx\right|
\end{equation}
recalling here that $f_r$ is the radial function given by
$$f_r(y)=\begin{cases}\frac{1}{4}& \text{ if }0\leq \|y\|<re^{(1-t)/4t} \\
\frac{1}{4t}-\log (\frac{\|y\|}{r})&\text{ if }re^{(1-t)/4t}\leq \|y\|\leq re^{1/4t}\\
0& \text{ else. }
\end{cases}$$
We note that in particular $f_r$ has derivative bounded by $C_r=e^{1/4}/r$. Tiling the support of $f_r$ by Vorono\"i cells of diameter the packing radius $\tau(\beta\calO^t)$, we can bound the error in approximating the Riemann integral via the lattice sum on each individual cell by $\Vol(\calO^t)\beta_p^{4t}\cdot C_r\cdot 2\tau(\beta\calO^t)$. For large enough $p,t$, we may estimate that the support of $f_r$ is covered by $\beta^{-4t}e\cdot \frac{Vol(\IB(r))}{\Vol(\calO^t)}$ cells (with an error that is $o(\Vol(\calO^t))$ as $t\to \infty$), so that we arrive at the total error estimate 
\begin{equation}
    \Delta(p,t)\lesssim 2eC_r\cdot \beta_p\frac{\Vol(\IB(r))}{\Vol(\calO^t)}\cdot \sqrt{t}\cdot \tau(\calO)
\end{equation}
We therefore obtain for $r$ satisfying \eqref{eq:roftsize} the bound as $p,t$ become large of 
$$\Delta(p,t)=O(t)\cdot p^{\frac{k-2t}{2t}}.$$
It now seems most convenient to take $k=t+1$ and we see that in particular the condition $t^2=o(p)$ suffices to guarantee for any given $\varepsilon>0$ that $\Delta(p,t)<\varepsilon$ for large enough rank $t$. We have thus shown that for any $\varepsilon$ we can find $t$ large enough so that under our assumptions on $p$ there exists $\Lambda\in \mathbb{L}_p$ with $\sum_{y\in\Lambda'}f_r(y)\leq (1-\varepsilon)\cdot 6$. The result now follows as in the proof of Theorem \ref{thm:improvedbounds}. 
\end{proof}
We therefore conclude:
\begin{corollary}\label{cor:computationalVance}
  Given any $0<\varepsilon<1$, for large enough $t$ a lattice $\tilde{\Lambda}$ in dimension $4t$ whose packing density satisfies
  $$\Delta(\tilde{\Lambda})\geq (1-\varepsilon)\cdot \frac{24t\zeta(4t)}{2^{4t}\cdot e(1-e^{-t})}$$
  can be constructed with $e^{4t^2\log (t)(1+o(1))}$ bit operations. 
  \end{corollary}
  \begin{proof}
    Given $\varepsilon$ and large enough $t$, it is easy to find a large enough prime $p$ so that the construction of Proposition \ref{prop:effectiveVance} applies. The corresponding family of lattices has $\vert \calC_{t+1}\vert$ elements in Proposition \ref{prop:effectiveVance}'s notation. They are generated by vectors with coefficients polynomial in $t$. The cost of computing their density or shortest vector as well as of computing their successive quaternionic minima (see remark \ref{rem:effectiveMinkowski}) is thus small compared to the cost of enumerating the family and we find by applying the effective version of \ref{prop:productminima} the desired lattice. 
  \end{proof}
  \section*{Acknowledgements}
  We would like to thank Gauthier Leterrier, Martin Stoller and Maryna Viazovska for helpful conversations on the topics of this paper. We thank Matthew De Courcy-Ireland for insightful comments on a previous version of this manuscript. 
  
 Nihar Gargava got funded by the Swiss National Science Foundation (SNSF), Project funding (Div. I-III), ``Optimal configurations in multidimensional spaces'', 184927.

\newpage
\begin{appendix}
  \section{Gram-Schmidt process for real semisimple algebras with positive involutions}
  \label{se:gs_process}

Let $A$ be a real semisimple algebra
and let $(\ )^{*}: A\rightarrow A$ be a positive involution. 

\subsection{Orthogonality in $A^k$}

Let $k \ge 1$. Fix a positive definite symmetric element $a \in A$. Let us define
\begin{align*}
  \langle \ , \ \rangle_{A} : A^{k} \times A^{k} & \mapsto A \\
  \langle x,y\rangle_{A} & = \sum_{i=1}^{k} x_{i} a y_{i}^*.
\end{align*}

  This form is $\mathbb{R}$-bilinear but not necessarily $A$-bilinear, since $A$ may not be commutative. We can however define a real positive definite symmetric bilinear form on $A^{k}$ given by 
  \begin{align*}
    \langle x,y\rangle_{\mathbb{R}} = \T\left( \langle x, y\rangle_{A}\right).
  \end{align*}
  
  We assume that $A^{k}$ carries this real inner product.

\begin{lemma}
  The following properties are satisfied by $\langle \ , \ \rangle_{A}$.
    \begin{enumerate}
  \item For $x,y \in A^{k}$, we have 
  \begin{align*}
    \langle x,y\rangle^{*}_{A} = \langle y,x\rangle_{A}.
  \end{align*}

    \item For $x \in A^{k} \setminus \{ 0\}$, $\langle x,x\rangle_{A}$ is a symmetric and positive definite in $A$, and hence is a unit in $A$.
    \item For $x,y \in A^{k}$ and $\alpha \in A$, we have \begin{align*}
	\langle \alpha x,y\rangle_{A} & = \alpha\langle x,y\rangle_{A},\\
	\langle  x, \alpha y \rangle_{A} & = \langle x ,y\rangle_{A} \alpha^{*}.
      \end{align*}
    \item Suppose $x \in A^{k}\setminus \{ 0\}$. Then there exists some $b \in A$ such that $\langle bx,bx\rangle = 1_{A}$.
  \end{enumerate}
\end{lemma}
\begin{proof}
  All of them are trivial verifications. For the last one, we must find a $b \in A$ such that $\langle x,x\rangle_{A}^{-1} = b^{*}b$. See \cite[Corollary 46]{gargava2021lattice}.
  
\end{proof}

\begin{definition}
  We say two vectors in $A^{k}$ are orthogonal if the above given product between them is $0$. 
  We call a set of vectors $\{ x_1,x_2,\cdots,x_m\} \in A^{k}$ orthonormal if $\langle x_i,x_{j}\rangle_{A} = \delta_{ij} 1_{A}$. 
\end{definition}

  Note that if $\langle x,y\rangle= 0$, then $\langle \alpha_1x,\alpha_2y\rangle=0$ for each $\alpha_1,\alpha_2 \in A$ and $x,y \in A^{k}$.

\begin{prop}
  Suppose $x_1,x_2,\cdots,x_m \in A^{k}$ is an orthonormal set of vectors.
  Then $m \le k$ and $x_1,x_2,\cdots,x_m$ are free under the left action of $A$. If $m=k$, then the vectors $x_1,\cdots,x_k$ freely generate $A^{k}$ as a left $A$-module.
\end{prop}
\begin{proof}
  Observe that if $a_1,a_2,\cdots, a_m \in A$ are such that 
  \begin{align*}
    & a_1 x_1 + a_2 x_2+ \dots + a_m x_m = 0 ,
  \end{align*}
  then we can just evaluate 
  \begin{align*}
    \langle a_1 x_1 + a_2 x_2+ \dots + a_m x_m  ,x_i \rangle  = a_i \langle x_i, x_i\rangle = a_i 1_{A} = 0.
  \end{align*}
  Hence, we have that each $a_i=0$. Using this, we get that 
  \begin{align*}
    A^{m} & \rightarrow A^{k} \\
    (a_1,a_2,\ldots,a_m) & \mapsto a_1x_1 + \cdots + a_m x_m.
  \end{align*}
  is an injective $\mathbb{R}$-linear map. Hence, for dimensional constraints, $m \le k$ and $m=k$ will imply that it is an isomorphism.
\end{proof}

\

\begin{remark}
  If we were to define
\begin{align*}
  \langle \ , \ \rangle_{A} : A^{k} \times A^{k} & \mapsto A \\
  \langle x,y\rangle & = \sum_{i=1}^{k} x_{i}^{*} y_{i},
\end{align*}
we would reach the same proposition above, but for right actions instead of left.
\end{remark}

\begin{corollary}
  Suppose $x_1,x_2,\cdots,x_k \in A^{k}$ is an orthonormal set of vectors. Then for any $v = a_1 x_1+ a_2 x_2 + \dots + a_k x_k$ for $\{ a_{i}\}_{i=1}^{k} \subseteq A$, we have 
  \begin{align*}
    \langle v,v\rangle_{\mathbb{R}} = \T(\langle v,v\rangle_{A}) = \sum_{i=1}^{k} \T(a_{i}^{*}a a_{i}).
  \end{align*}
\end{corollary}

\subsection{Gram-Schmidt algorithm}

The following algorithm is an analogue of Gram-Schmidt.

Suppose $v_1,v_2,\cdots,v_k \in A^{k}$ are some vectors that freely generate $A^{k}$ as a left $A$-module. We claim that using these vectors, it is possible to create a set of orthonormalized basis vectors $x_1,x_2,\cdots,x_k$.

First, we do the following. Define for $u,v \in A^{k}$,
\begin{align*}
  \pr(u, v)= 
  \begin{cases}
    {\langle u,v\rangle_{A}}\langle u,u\rangle_{A}^{-1} u & \text{ if }u \neq 0 \\
    0 & \text{if } u = 0
  \end{cases}.
\end{align*}
This has the property that $\langle \pr(u,v),u\rangle_{A} = \langle u,v\rangle_{A}$.

Generate vectors $x_1',x_2',\cdots,x_k'$ as follows.
  \begin{align*}
    x_1'&  = {v_1}, \\ 
    x_{2}'& = v_{2} - \pr(x_1',v_2) \\
    x_{3}'& = v_{3} - \pr(x_1',v_3) - \pr(x_2',v_3)\\
    x_{4}'& = v_{4} - \pr(x_1',v_4) - \pr(x_2',v_4) - \pr(x_3',v_4)\\
    \vdots\\
  \end{align*}
  
  We can prove that $\langle x_i',x_j'\rangle =0$ for $i > j$ by first induction via ordering $(i,j)$ as 
  $(2,1),(3,1),\cdots$, $(k,1),(3,2),(4,2),\cdots$, $(k,2),(4,3),\cdots$. Now choose $b_i$ such that $x_i = b_i x'_i$ has $\langle x_i,x_i\rangle = 1_{A}$. Hence, we are done.
  
\begin{definition}
Given a real semisimple algebra $A$ with a positive involution $(\ )^{*}$, we call the above method of generating $x_1,x_2,\cdots,x_m$ from vectors $v_1,v_2,\cdots,v_m$ the Gram-Schmidt algorithm. If $\{ v_i\}_{i=1}^{m}$ are free under left $A$ multiplication, then so are $\{ x_i\}_{i=1}^{m}$.
\end{definition}
  
\end{appendix}

\nocite{}
\bibliography{divrings.bib}

\end{document}